\documentclass{amsart}

\usepackage{amssymb}
\usepackage{amsmath}
\usepackage{mathtools}
\usepackage{microtype}
\usepackage{csquotes}
\usepackage[mathscr]{euscript}
\usepackage{stmaryrd}
\usepackage{xspace}

\usepackage[linktoc=all]{hyperref}
\usepackage[hyphenbreaks]{breakurl}

\usepackage{amssymb}
\usepackage{latexsym}

\usepackage{comment}

\usepackage{graphicx}

\usepackage[colorinlistoftodos, shadow]{todonotes}

\usepackage[
style=alphabetic,
backend=biber,
maxnames=5,
maxalphanames=5,
language=english,
doi=false,
backref=true,
backrefstyle=three,
isbn=false,
url=true
]{biblatex}

\bibliography{Bibliography}

\usepackage{tikz}
\usetikzlibrary{matrix,arrows}
\usepackage{tikz-cd}

\usepackage[makeroom]{cancel}
\definecolor{cross}{RGB}{140,140,140}
\newcommand\ccancel[2][cross]{\renewcommand\CancelColor{\color{#1}}\xcancel{#2}}

\usepackage[all]{xy}

\usepackage{setspace}

\newenvironment{changemargin}[2]{%
\begin{list}{}{%
\setlength{\topsep}{0pt}%
\setlength{\leftmargin}{#1}%
\setlength{\rightmargin}{#2}%
\setlength{\listparindent}{\parindent}%
\setlength{\itemindent}{\parindent}%
\setlength{\parsep}{\parskip}%
}%
\item[]}{\end{list}}

\newcounter{temporaryResumeCounter}

\theoremstyle{plain}
\newtheorem{thm}{Theorem}[section]
\newtheorem{lem}[thm]{Lemma}
\newtheorem{thma}{Theorem}

\newtheorem{prop}[thm]{Proposition}
\newtheorem{cor}[thm]{Corollary}

\newtheorem{remark}[thm]{Remark}

\theoremstyle{definition}
\newtheorem{defn}[thm]{Definition}

\newtheorem{rem}[thm]{Remark}
\newtheorem{example}[thm]{Example}

\newcommand{\pref}[2]{\hyperref[#2]{#1 \ref*{#2}}}
\hypersetup{urlcolor=blue, citecolor=blue, linkcolor=blue, colorlinks=true}

\newcommand{\bbE}{{\mathbb{E}}}

\newcommand{\bbH}{{\mathbb{H}}}
\newcommand{\bbI}{{\mathbb{I}}}

\newcommand{\bbN}{{\mathbb{N}}}

\newcommand{\bbP}{{\mathbb{P}}}
\newcommand{\bbQ}{{\mathbb{Q}}}
\newcommand{\bbR}{{\mathbb{R}}}

\newcommand{\bbZ}{{\mathbb{Z}}}

\newcommand{\calA}{{\mathcal{A}}}

\newcommand{\calC}{{\mathcal{C}}}

\newcommand{\calF}{{\mathcal{F}}}

\newcommand{\calM}{{\mathcal{M}}}

\newcommand{\calR}{{\mathcal{R}}}
\newcommand{\calS}{{\mathcal{S}}}

\newcommand{\frakD}{{\mathfrak{D}}}

\newcommand{\frakS}{{\mathfrak{S}}}

\newcommand{\rmT}{{\mathrm{T}}}

\let\ORGvarepsilon=\varepsilon
\let\varepsilon=\epsilon
\let\epsilon=\ORGvarepsilon
\newcommand{\two}{{\mathrm{I\!I}}}

\newcommand{\scal}{{\mathbf{scal}}}
\newcommand{\sect}{{\mathbf{sec}}}
\newcommand{\ric}{{\mathbf{Ric}}}

\newcommand{\diff}{{\mathrm{Diff}}}

\newcommand{\topo}{\mathrm{Top}}
\newcommand{\Spin}{\mathrm{Spin}}
\newcommand{\SO}{\mathrm{SO}}
\newcommand{\boc}{B\mathrm{O}(d)\langle c-1\rangle}
\newcommand{\ort}{\mathrm{O}}
\newcommand{\sfphi}[1]{\calS_{\calF,\varphi_{#1}}}

\newcommand{\hocolim}[1]{{\underset{#1}{\mathrm{hocolim}}}\ }

\newcommand{\im}{\mathrm{im\ }}
\newcommand{\tr}{\mathrm{\textbf{tr}}}
\newcommand{\rk}[1]{\mathrm{{rank}(#1)}}

\newcommand{\codim}{\mathrm{codim\ }}

\newcommand{\pr}{{\mathrm{pr}}}

\newcommand{\congarrow}{\overset{\cong}\longrightarrow}

\newcommand{\haut}{\mathrm{h}\mathbf{Aut}}

\newcommand{\embeds}{\hookrightarrow}

\newcommand{\too}{\longrightarrow}

\newcommand{\res}{\mathrm{res}}
\newcommand{\stab}{\mathrm{stab}}
\newcommand{\inddiff}{\mathrm{inddiff}}

\newcommand{\op}{\mathrm{op}}

\newcommand{\st}{\mathrm{st}}
\newcommand{\id}{\mathrm{id}}
\newcommand{\mfd}{\mathrm{Mfd}}

\newcommand{\ko}{\mathrm{KO}}
\newcommand{\lin}{\mathbf{Lin}}
\newcommand{\cl}{\mathrm{Cl}}
\newcommand{\fred}{\mathrm{Fred}}
\newcommand{\kom}{\mathbf{Kom}}
\newcommand{\scpr}[1]{\langle#1\rangle}
\newcommand{\ind}{\mathrm{Ind}}
\newcommand{\dirac}{\mathrm{Dir}}
\newcommand{\cyl}{\mathrm{Cyl}}
\newcommand{\hofib}{\mathrm{hofib}}
\newcommand{\cst}{\mathrm{const}}
\newcommand{\mt}{\mathrm{MT}}
\newcommand{\mtt}{\mathrm{MT\theta}}
\newcommand{\pt}{\mathrm{pt}}
\newcommand{\dt}{\mathrm{d}t}
\newcommand{\Th}{\mathrm{Th}}
\newcommand{\actson}{\curvearrowright}
\DeclarePairedDelimiter\floor{\lfloor}{\rfloor}

\newcommand{\bas}{\mathrm{bas}}
\newcommand{\trg}{\mathrm{trg}}

\newcommand{\ie}{\mbox{i.\,e.\,}{}}

\newcommand{\psc}{{\mathrm{scal}>0}}
\newcommand{\psec}{{\mathrm{sec}>0}}
\newcommand{\prc}{{\mathrm{Ric}>0}}
\newcommand{\ppc}{{p\text{-}\mathrm{curv}>0}}
\newcommand{\pkrc}{{k\text{-}\mathrm{Ric}>0}}
\newcommand{\pdkrc}{{(d-k)\text{-}\mathrm{Ric}>0}}

\newcommand{\AlgCurvOp}[1][d]{{\mathcal C_{\mathrm B}(\bbE^{#1})}}

\usepackage{tcolorbox}
\tcbuselibrary{skins,breakable}

\DeclareRobustCommand{\SkipTocEntry}[9]{}

\subjclass[2010]{53C21, 57R20, 57R65, 57R90, 58D05, 58D17.}

\begin{document}

\title[Spaces of positive Intermediate Curvature Metrics]{Spaces of positive Intermediate curvature metrics}

\author{Georg Frenck}
\thanks{G.F. was supported by Deutsche Forschungsgemeinschaft (DFG, German Research Foundation) under 281869850 (RTG 2229).}
\address{KIT, Karlsruher Institut für Technologie\\
Englerstra\ss e 2\\
76131 Karlsruhe\\
Bundesrepublik Deutschland}
\email{math@frenck.net}
\email{georg.frenck@kit.edu}
\urladdr{Frenck.net/Math}

\author{Jan-Bernhard Korda\ss}
\thanks{J.K. was supported by the SNSF-Project 200021E-172469 and the DFG-Priority programme \emph{Geometry at infinity} (SPP 2026).}
\address{}
\email{jb@kordass.eu}
\urladdr{http://kordass.eu}

\begin{abstract}
In this paper we study spaces of Riemannian metrics with lower bounds on intermediate curvatures. We show that the spaces of metrics of positive $p$-curvature and $k$-positive Ricci curvature on a given high-dimensional $\Spin$-manifold have many non-trivial homotopy groups provided that the manifold admits such a metric.
\end{abstract}

\maketitle

\section{Introduction}

\noindent

\noindent 
Given a compact manifold $M$ with possibly nonempty boundary, the classification of Riemannian metrics on $M$ satisfying a given curvature condition is a central problem in Riemannian geometry. In the present article we will study the uniqueness question. Of course, open conditions like positive scalar, Ricci or sectional curvature are preserved under small perturbations of a metric and so there cannot be a unique metric satisfying them. Therefore it is more reasonable to study uniqueness \enquote{up to continuous deformation}, which translates into the following question: 
\vspace{3pt}
\begin{changemargin}{2cm}{2cm}
\begin{center}\textit{Is the space of Riemannian metrics on $M$ satisfying a given curvature condition contractible?}\end{center}
\end{changemargin}
\vspace{3pt}
\noindent In recent years, a lot of effort has gone into understanding the homotopy type of the space $\calR_\psc(M)_h$ of metrics of positive scalar curvature which restrict to $h+\dt^2$ in a collar neighborhood of the boundary. For example, Botvinnik--Ebert--Randal-Williams in \cite{berw} have studied this space for $d$-dimensional $\Spin$-manifolds using the secondary index-invariant $\inddiff$ which is a well-defined homotopy class of a map
\[\inddiff\colon\calR_\psc(M)_h\times\calR_\psc(M)_h\to\Omega^{\infty+d+1}\ko\]
first defined by Hitchin in \cite{hitchin_spinors}. Fixing a base-point $g\in\calR_\psc(M)_h$ one obtains a homotopy class of a map $\inddiff_g\colon\calR_\psc(M)_h\to\Omega^{\infty+d+1}\ko$ and they showed that this induces a nontrivial map
\[A_{m-1}(M,g)\colon\pi_{m-1}(\calR_\psc(M)_h)\too\ko^{-d-m}(\pt)\cong\begin{cases}\bbZ &\text{ if } d+m\equiv0(4)\\\bbZ/2 &\text{ if } d+m\equiv1,2(8)\end{cases}\]
on homotopy groups, provided that $d\ge 6$, $M$ admits a $\Spin$-structure and the target is nontrivial. This shows that the space $\calR_\psc(M)_h$ is at least as complicated as the infinite loop space of the real $K$-theory spectrum. For more results on this space see \cite{walsh_hspaces,crowleyschicksteimle, erw_psc2}.

In this paper we generalize the main result from \cite{berw} to a greater class of curvature conditions. The most prominent examples of these are given by two of the intermediate curvature conditions, namely \emph{positive $p$-curvature} and \emph{$k$-positive Ricci curvature}, for precise definitions see \pref{Section}{sec:curvatureconditions}. 

The notion of \enquote{p-curvature is an extension of the scalar curvature proposed by Gromov} \cite[p.301]{labbi_stability}. It interpolates between scalar and sectional curvature and  has been studied for example in \cite{labbi_actions,botvinnik-labbi}). For $p\ge0$ let $\calR_\ppc(M)_h\subset \calR_\psc(M)_h$ denote the subspace of metrics of positive $p$-curvature.

\begin{thma}\label{thm:mainpcurv}
Let $p\ge0$ and let $M$ be a $\Spin$-manifold of dimension $d\ge6+2p$. Let $g\in\calR_{\ppc}(M)_h$. Then for all $m\ge1$ such that $d+m\equiv0(4)$ the composition
\[\pi_{m-1}(\calR_\ppc(M)_h)\too\pi_{m-1}(\calR_\psc(M)_h)\overset{A_{m-1}(M,g)}\too\ko^{-d-m}(\pt)\cong \bbZ\]
is nontrivial.
\end{thma}

\noindent In \cite{wolfson} Wolfson introduced the notion of $k$-positive Ricci curvature which has been studied for example in \cite{crowleywraith, walshwraith}. A manifold is said to have $k$-positive Ricci curvature if the sum of the $k$ smallest eigenvalues of the Ricci curvature is positive. This gives an interpolation between positive scalar curvature being $d$-positive Ricci curvature and positive Ricci curvature which is $1$-positive Ricci curvature. For $1\le k\le d$ let $\calR_\pkrc(M)_h\subset \calR_\psc(M)_h$ denote the subspace of metrics of $k$-positive Ricci curvature. For technical reasons it is more natural to state our result for $(d-k)$-positive Ricci curvature instead of $k$-positive Ricci curvature.

\begin{thma}\label{thm:maindkric}
Let $k\ge1$ and let $M$ be a $\Spin$-manifold of dimension $d\ge 4+2k$. Let $g\in\calR_{\pdkrc}(M)_h$. Then for all $m\ge1$ such that $d+m\equiv0(4)$ the composition
\[\pi_{m-1}(\calR_\pdkrc(M)_h)\too\pi_{m-1}(\calR_\psc(M)_h)\overset{A_{m-1}(M,g)}\too\ko^{-d-m}(\pt)\]
is nontrivial.
\end{thma}

\noindent In the spirit mentioned above, these results can be paraphrased by saying that the spaces $\calR_\ppc(M)_h$ and $\calR_\pdkrc(M)_h$ are at least as complicated as the infinite loop space of the real $K$-theory spectrum, provided that $M$ is $\Spin$ and the dimension of $M$ is big enough.

\begin{rem}[State of the art]\leavevmode
	\begin{enumerate}
		\item The corresponding results in degrees $d+m\equiv 1,2(8)$ can also be shown by our methods. This however is already known by the work of Crowley--Schick--Steimle \cite{crowleyschicksteimle} for all $d\ge6$. They showed that the orbit map $\rho\colon f\mapsto f^*g$ induces for every $g\in\calR_\psc(D^d)_h$ a surjective map
	\[\pi_{m-1}(\diff_\partial(D^d))\overset{\rho}\too \pi_{m-1}(\calR_\psc(D^d)_h) \too\ko^{-d-m}(\pt)=\bbZ/2.\]
	By extending diffeomorphisms by the identity, we get a map $\diff_\partial(D^d)\to\diff(M)$ and the result above holds if $D^d$ is replaced by any $\Spin$-manifold $M$ of positive scalar curvature. Furthermore, since the orbit map factors through any $\diff_\partial(D^d)$-invariant subspace of $\calR_\psc(D^d)_h$, the results from \cite{crowleyschicksteimle} are true for \emph{any} curvature condition that implies positive scalar curvature and is satisfied by $M$. \cite{crowleyschicksteimle} is a strict generalization of the results from \cite{hitchin_spinors, crowleyschick}.
		\item In \cite{KKRW}, Krannich--Kupers--Randal-Williams showed that the image of the orbit map $\pi_3(\diff(\bbH\bbP^2))\to\pi_3(\calR_\psec(\bbH\bbP^2))\embeds\pi_3(\calR_\ppc(\bbH\bbP^2))$ contains an element of infinite order for every $p\ge0$. Furthermore, the rational homotopy type of $\diff(M)$-invariant subspaces of $\calR_\psc(M)$ has been studied by Reinhold and the first named author in \cite{a-hat-bundles}. Here it is shown that this space has non-vanishing higher rational cohomology, provided that $M$ is a high-dimensional $\Spin$-manifold and given by $N\#(S^p\times S^q)$ for $p,q$ in a range. This is a generalization of the main result from \cite{bew}. To the best of our knowledge, those are the only other known result about non-triviality of the higher rational homotopy type of spaces of positive $p$-curvature metrics (resp. $k$-positive Ricci curvature metrics) for $p\ge1$ (resp. $k\le d-1$).
		\item Concerning $k$-positive Ricci curvature, there is one other result besides \cite{KKRW} and \cite{a-hat-bundles} we would like to mention. Namely, Walsh--Wraith have shown in \cite{walshwraith} that for $d\ge3$ and $k\ge2$ the space $\calR_\pkrc(S^d)$ is an $H$-space and the component of the round metric is in fact a $d$-fold loop space.
	\end{enumerate}
\end{rem}

\noindent The present article grew out of an attempt to extract the necessary geometric ingredients from \cite{berw}. The main one is a parametrized version of the famous Gromov--Lawson--Schoen--Yau surgery theorem \cite{gromovlawson_classification, schoenyau} which is due to Chernysh \cite{chernysh} and has been first published by Walsh \cite{walsh_cobordism}, see also \cite{ebertfrenck}. It states that the homotopy type of $\calR_\psc(M)$ is invariant under surgeries with certain dimension and codimension restrictions. It turns out, that the above theorems follow from a more general result about so-called \emph{surgery-stable} $\diff(M)$-invariant subsets $\calF(M)\subset\calR_\psc(M)$ with a few extra properties. We will give the general statement of our main result \pref{Theorem}{thm:maingeneral} in the course of \pref{Section}{sec:prelim}, after we introduced the relevant notions.

\addtocontents{toc}{\SkipTocEntry}
\subsection*{Outline of the argument}
Let $\calF(M)_h\subset\calR_\psc(M)_h$ be a $\diff_\partial(M)$-invariant subset. The strategy for proving \pref{Theorem}{thm:mainpcurv} and \pref{Theorem}{thm:maindkric} for manifolds of dimension $2n$ is to construct maps $\rho\colon\Omega^{\infty+1}\mtt_{c-1}(2n)\to \calF(M)_h$ from the infinite loop space of the Madsen--Tillmann--Weiss spectrum $\mtt_{c-1}(2n)$ associated to $\theta$ the tangential $(c-1)$-type of $M$ (cf.\ \pref{Section}{sec:mtw} for the definition). Afterwards one has to show that the composition with the maps from those theorems is weakly homotopic to the loop map of $\hat\calA\colon\Omega^{\infty}\mtt_{c-1}(2n)\to\Omega^{\infty+2n}\ko(\pt)$ which is accomplished by index theoretic arguments from \cite{berw}. Computations then show that $\Omega\hat\calA$ induces a surjection on rational homotopy groups, whenever the target is nontrivial.

The construction is first done for $M$ a certain $\theta$-nullcobordism of $S^{2n-1}$ which itself is $\theta$-cobordant to the disk relative to the boundary. By gluing in $k$ copies of $K\coloneqq ([0,1]\times S^{2n-1})\#(S^n\times S^n)$ along the boundary, we obtain the manifold $M_k\coloneqq M\cup k\cdot K$. We will show that there is a metric $g_{\st}\in\calF(K)_{h_\circ,h_\circ}$ for $h_\circ$ the round metric on $S^{2n-1}$ with the property, that the map $\calF(W)_{h_N,h_\circ}\to \calF(W\cup K)_{h_N,h_\circ}$ gluing in $g_\st$ is a homotopy equivalence for any cobordism $W\colon N\to S^{2n-1}$ and any metric $h_N\in\calR(N)$. Therefore $\calF(M_k)_{h_\circ}\embeds\calF(M)_{h_\circ}$ and in particular
\begin{equation}\label{eq:hocolim}
	\calF(M)_{h_\circ}\to \hocolim{k\to\infty}\calF(M_k)_{h_\circ}
\end{equation}
are homotopy equivalences. Consider the Borel construction
\[\calF(M)_{h_\circ}\too \underbrace{E\diff_\partial(M)\underset{\diff_\partial(M)}\times \calF(M)_{h_\circ}}_{\eqqcolon \calF(M)_{h_\circ}\sslash\diff_\partial(M)}\too B\diff_\partial(M).\]
Since there are stabilization maps $\calF(M_k)_{h_\circ}\to\calF(M_{k+1})_{h_\circ}$ and $\diff_\partial(M_k)\to\diff_\partial(M_{k+1})$ we get stabilization maps for the associated Borel constructions and after passing to the (homotopy) colimit, this yields the following fibration:
\begin{equation}\label{eq:fibration}
	p_\infty\colon\hocolim{k\to\infty}\calF(M_k)_{h_\circ}\sslash\diff_\partial(M_k)\to\hocolim{k\to\infty}B\diff_\partial(M_k)
\end{equation}
The space $\hocolim{}_{k\to\infty}B\diff_\partial(M_k)$ admits an acyclic map to
\[\Psi\colon \hocolim{k\to\infty}B\diff_\partial(M_k)\to\Omega^\infty\mtt_{c-1}(2n)\]
by the work of Galatius--Randal-Williams \cite{grw_stable}. By an obstruction argument the fibration from \eqref{eq:fibration} extends to a fibration $p_\infty^+\colon T^+\to \Omega^\infty\mtt_{c-1}(2n)$, meaning that the associated diagram of fibrations
\begin{center}
\begin{tikzpicture}
	\node (0) at (0,0) {$\hocolim{k\to\infty}B\diff_\partial(M_k)$};
	\node (1) at (5,0) {$\Omega^\infty\mtt_{c-1}(2n)$};
	\node (2) at (0,1.5) {$\hocolim{k\to\infty}\calF(M_k)_{h_\circ}\sslash\diff_\partial(M_k)$};
	\node (3) at (5,1.5) {$T^+$};
	
	\draw[->] (0) to node[auto]{$\Psi$} (1);
	\draw[->] (3) to node[auto]{$p_\infty^+$} (1);
	\draw[->] (2) to node[auto]{$p_\infty$} (0);
	\draw[->] (2) to (3);
\end{tikzpicture}
\end{center}
is homotopy-cartesian, \ie a homotopy pullback diagram. The main input for solving this obstruction problem is the fact that the pullback action $\diff_\partial(M_k)\actson\calF(M_k)_{h_\circ}$ factors up to homotopy through an abelian group for all $k$, which follows from surgery-stability combined with an argument in the style of Eckmann--Hilton. The desired map $\rho$ is then given by the fiber transport map associated to the fibration $p_\infty^+$ composed with the homotopy-inverse of the stabilization map from (\ref{eq:hocolim}). Using the additivity theorem for the index, this result is the propagated from $M$ to \emph{any} $\Spin$-manifold of the same dimension. Jumping to the next dimension requires the spectral flow index theorem and the additional assumption that the map $\calF(M)\to \calF(\partial M)$ restricting a metric to the boundary to be a fibration.

\addtocontents{toc}{\SkipTocEntry}
\subsection*{Outline of the paper} 
In \pref{Section}{sec:prelim} we develop the basic notions needed in this paper, starting with the definition of Riemannian functors in \pref{Section}{sec:riemann}. These will be contravariant functors on the category of manifolds with codimension $0$ embeddings to the category of spaces, assigning to a manifold a subspace of Riemannian metrics. The main examples are given by curvature conditions, which is reviewed in \pref{Section}{sec:curvatureconditions} where we also give precise definitions of the intermediate curvature conditions. Afterwards we introduce the notions of \emph{surgery-stability} and \emph{fibrancy} for Riemannian functors in \pref{Section}{sec:surgery-stable} and \pref{Section}{sec:fibrant}. We give a list of Riemannian functors satisfying these two conditions after proving a criterion for fibrancy. In \pref{Section}{sec:thmmain} we are finally able to state the general version of our main result. The computations of the image of the map $\hat\calA\otimes\bbQ$ mentioned above is then carried out in \pref{Section}{sec:mtw}, where Madsen--Tillmann--Weiss spectra are introduced. The final \pref{Section}{sec:index-theory} of the preliminaries is a recollection of the index-theoretic arguments from \cite{berw} involved in the proof of our main result, which we included to give some context.

In \pref{Section}{sec:main} we carry out the proof of our main theorem. In \pref{Section}{sec:diffeo-action} we show that the pullback action factors through an abelian group which builds the basis for the obstruction argument used in \pref{Section}{sec:construction} to construct the map $\rho$ mentioned above. Afterwards we deduce the propagation result in \pref{Section}{sec:propagation} which enables to extend the result from a particular manifold to all of them. For convenience we show how the proof of our main result assembles in \pref{Section}{sec:assemble}.

We close this paper by giving an overview of other recent results about the homotopy type of $\calR_\psc(M)$ in \pref{Section}{sec:epilog}. The proofs of those also depend mainly on the parametrized surgery theorem from a geometrical point of view. We believe that many of them can also be generalized to hold for positive $p$-curvature and $k$-positive Ricci curvature, too.

\tableofcontents

\section{Preliminaries}\label{sec:prelim}

\subsection{Riemannian functors}\label{sec:riemann}
  Let $M$ be a smooth compact manifold with (possibly empty) boundary $\partial M$.
  If $\partial M\not=\emptyset$ we will always assume that $M$ is equipped with a collar, i.e.\ with an embedding $c \colon \partial M \times [0,1) \to M$ such that $\{0\} \times \partial M$ is canonically identified with $\partial M \subset M$.
  We denote by $\calR(M)$ the set of all smooth Riemannian metrics $g$ on $M$, which additionally satisfy $(c|_{\partial M \times [0,\varepsilon)})^{*}g = g|_{\partial M} + \mathrm d t^2$ for some $\varepsilon>0$.
  Hence, the metrics on a manifold with boundary are assumed to be of product form within a collar region of the boundary.
  We endow $\calR(M)$ with the $C^{\infty}$-topology, i.e.\ the subspace topology of the Fréchet topological space of smooth, symmetric $(0,2)$-tensor fields on $M$.
The diffeomorphism group $\diff(M)$ of $M$ acts on $\calR(M)$ by push-forward of Riemannian metrics, i.e.\ via $\diff(M) \times \calR(M) \to \calR(M),\ (f, g) \mapsto (f^{-1})^*g$.

Let $\mfd$ denote the category which has compact manifolds with (possibly empty) boundary as objects and morphisms are given by smooth $\codim 0$-embeddings.

\begin{defn}\label{defn:Riemannian-functor}
	A functor $\calF \colon \mfd\to\topo^\op$ is called \emph{Riemannian} if $\calF(M) \subset \calR(M)$, $\calF(f) = f^*\colon \calF(N)\to\calF(M)$ and the canonical homeomorphism $\calR(M)\times\calR(N)\to \calR(M\amalg N)$ restricts to a homeomorphism $\calF(M)\times\calF(N)\to \calF(M\amalg N)$.
\end{defn}

\begin{remark}
\begin{enumerate}
	\item Since diffeomorphisms are $\codim 0$-embeddings, $\calF(M)$ is a $\diff(M)$-invariant subset of $\calR(M)$.
	\item The pull-back of Riemannian metrics along a smooth embedding is a continuous map with respect to the $C^{\infty}$-topology on the spaces of Riemannian metrics.
\end{enumerate}
\end{remark}

\begin{defn}
	We say that a Riemannian functor $\calF$ implies positive scalar curvature, if $\calF(M)\subset\calR_\psc(M)$ for every manifold $M$.
\end{defn}

\begin{example}
  \begin{enumerate}
  \item One of the most studied examples for a Riemannian functor arises from positive scalar curvature metrics, i.e.\ by the assignment
    \begin{align*}
      \calR_{\psc} \colon M \mapsto \{ g \in \mathcal R(M) \mid \scal(g) > 0 \},
    \end{align*}
    where $\scal(g) \colon M \to \mathbb R$ denotes the scalar curvature function of the metric $g$.
    It is immediately clear that for $g \in \calR_{\psc}(N)$ and a $\codim 0$-embedding $f \colon M \to N$ the pull-back $f^{*}g$ is a metric of positive scalar curvature on $M$.
  \item Clearly, this example can be extended to more general (open) curvature conditions, which we will recall in the subsequent section.
    Note, however, that for the most common conditions \enquote{positive Ricci curvature} and \enquote{positive sectional curvature} on a manifold with non-empty boundary $M$, the space $\calF(M)$ is empty. This is implied by our assumption on boundary collars, since the cylindrical metric $g + \mathrm d t^2$ on $\partial M \times \bbR$ has neither positive Ricci, nor positive sectional curvature.
  \end{enumerate}
\end{example}

\begin{defn}
	A Riemannian functor is called
	\begin{itemize}
		\item\emph{open} if for every manifold $M$ the space $\calF(M) \subset \mathcal R(M)$ is an open subspace.
		\item\emph{cellular} if for every manifold $M$ the space $\calF(M)$ is dominated by a $CW$-complex.
	\end{itemize}
\end{defn}

\begin{remark}
	An open Riemannian functor $\calF$ is cellular by \cite[Theorem 13]{palais_infinite-dimensional}. Recall that for $CW$-dominated spaces a weak homotopy equivalence is an actual homotopy equivalence by Whitehead's theorem.
\end{remark}

\subsection{Curvature conditions}\label{sec:curvatureconditions}

Let $(M^d,g)$ be a Riemannian manifold of dimension $d$.
Recall that different notions of curvature of the metric $g$ at a given point $p$ are encoded in the Riemann curvature operator $R_p$ at $p$.
Any choice of an orthonormal basis in the tangent space $\rmT_pM$ yields a description of $R_p$ in terms of a self-adjoint endomorphism on $\bigwedge^2 \mathbb E^d$, where $\mathbb E^d$ denotes the euclidean inner product space.
This object lies in the vector space of algebraic curvature operators $\AlgCurvOp[d]$, which consists of all self-adjoint endomorphisms of $\bigwedge^2 \mathbb E^d$ satisfying the Bianchi identity (cf.\ \cite[p.45ff]{besse}).
Changing orthonormal bases gives rise to an action of $\operatorname{O}(d)$ and subsets $C \subset \AlgCurvOp[d]$ invariant under this action are referred to as \emph{curvature conditions}.
We say that a Riemannian metric $g$ on a smooth manifold $M$ \emph{satisfies} a curvature condition $C \subset \AlgCurvOp[d]$,
if for every point $p \in M$ the description of $R_p$ in terms of an orthonormal basis in $\rmT_pM$ is contained in $C$.

Let $d_C\ge0$ and let $C = \{C_d\}_{d \ge d_C}$ with $C_d \subset \AlgCurvOp[d]$ be a sequence of curvature condition. We define a Riemannian functor
\begin{align*}
  \mathcal R_C \colon M^d \mapsto \{ g \in \calR(M) \mid g \text{ satisfies } C_d \}.
\end{align*}
Our convention will be that $\mathcal R_C \colon M^d \mapsto \emptyset$ for all $M^d \in \mfd$ with $0 \leq d < d_C$. As can be seen from the following examples, $d_C$ can be thought of as the lowest dimension in which it makes sense to consider the curvature condition $C$.

\begin{example}
\begin{enumerate}
\item There exist corresponding subsets to all classical curvature bounds, e.g.\ bounds on the sectional, Ricci or scalar curvature.
  For example, we can express (globally point-wise) positive sectional, Ricci and scalar curvature as conditions
  \begin{align*}
    (\psec)_d &:= \{ R \in \AlgCurvOp[d] \mid \sect(R) > 0\},\\
    (\prc)_d &:= \{ R \in \AlgCurvOp[d] \mid \ric(R) > 0\},\\
    \operatorname{psc}_d := (\psc)_d &:= \{ R \in \AlgCurvOp[d] \mid \tr(R) > 0 \}.
  \end{align*}
  Here we write $\sect(R)(X,Y) := \langle R(X \wedge Y), X \wedge Y\rangle$ for $X,Y$ an orthonormal basis of a 2-plane in $\mathbb E^d$, $\ric(R)(X) = \sum_{i=2}^d\sect(X,E_i)$ for $(X,E_2,\dots,E_d)$ an orthonormal basis of $\bbE^d$ and $\tr(R)$ denotes the trace of the algebraic curvature operator, which coincides with its scalar curvature up to a factor of $\frac{1}{2}$. In these cases we have $d_{\psec} = d_{\prc} = d_{\psc} = 2$.
\item The notion of $p$-curvature, where $p$ is an integer, was proposed by Gromov (cf.\ \cite[p.301]{labbi_stability}) and is a natural generalization of scalar and sectional curvature which provides an interpolation between both. Let $(M,g)$ be a Riemannian manifold of dimension $d \geq p + 2$ and let $\mathrm G_p(\rmT M)$ denote the $p$-Graß\-mannian bundle over $M$ and $U(P_x^{\perp})$ be a neighborhood around $0$ in the plane perpendicular to a $p$-plane $P_x\subset T_xM$. The map
    \begin{align*}
      s_p \colon \mathrm G_p(\rmT M) \to \bbR
      \qquad
      P_x \mapsto \scal_x\bigr(\exp_x(U(P_x^{\perp}))\bigr)
    \end{align*}
    is referred to as \emph{$p$-curvature function}.
    If $s_p$ is positive on all of $\mathrm G_p(\rmT M)$, the metric $g$ is said to have \emph{positive $p$-curvature}.

  The term $p$-curvature coincides with scalar curvature for $p = 0$ and with (the double of) sectional curvature for $p = d-2$. Without much effort, one can show that positive $p$-curvature implies positive $(p-1)$-curvature and thus ultimately positive scalar curvature. If $\{E_i\}_{1 \leq i \leq d-p}$ is an orthonormal basis of $P_x^{\perp} \subset \rmT_xM$ we have $s_p(P_x) = \sum_{i,j = 1}^{d-p} \sect(E_i,E_j)$, where for convenience we set $\sect(E_i,E_i) := 0$. It is easy to see in this description that $s_1(\operatorname{span}(v)) = s_0 - 2\ric(v)$ for any element $v \in \operatorname S(\rmT_xM)$, which is precisely double the value of the Einstein tensor $E(v,v) = \frac{1}{2}\scal g(v,v) - \ric(v)$.

  Positive $p$-curvature can be described as a curvature condition given by an open convex cone
  \begin{align*}
    (p\text{-curv} > 0)_d := \{ R \in \AlgCurvOp[d] &  \mid s_p(R)(P) > 0 \\
                                                                          & \quad \forall P \leq \mathbb E^d \text{ with } \dim P = p\},
  \end{align*}
  where $s_p(R) \colon G_p(\bbR^d) \to \bbR$ is the map $P \mapsto \sum_{i,j=1}^{d-p}\sect(R)(E_i,E_j)$ for an orthonormal basis $\{E_i\}$ of $P^{\perp}$. Hence for every fixed $p \geq 0$, we obtain a sequence $C_d := (p\text{-curv} > 0)_d$ of curvature conditions that yield a Riemannian functor as above with $d_{p\text{-curv} > 0} = p + 2$.

\item In \cite{wolfson} J.\ Wolfson introduced the notion of $k$-positive Ricci curvature, which interpolates between positive scalar curvature (for $k = d$) and positive Ricci curvature (for $k = 1$). A Riemannian metric $g$ on a manifold $M^d$ of dimension $d\ge3$ is said to have \emph{$k$-positive Ricci curvature} for $1 \leq k \leq d$, if the eigenvalues $k_1 \leq \ldots \leq k_d$ of the Ricci curvature satisfy $\sum_{i = 1}^k k_i > 0$. This defines a curvature condition given by an open convex cone
  \begin{align*}
    (\pkrc)_d := \{ R \in \AlgCurvOp[d] & \mid \sum_{i=1}^k \ric(R)(E_i) > 0 \\\
                                                                             & \quad \forall \{ E_1, \ldots, E_k\} \text{ orthonormal in } \mathbb E^d \}.
  \end{align*}
  For technical reasons it is more convenient for us to replace $k$ by $(d-k)$ (cf.\ \pref{Remark}{rem:dk-ricci}). For fixed $k\ge 0$ we define a sequence of curvature conditions $C_d := ((d-k)\text{-pos Ric})_d$ and obtain a Riemannian functor with $d_{(d-k)\text{-pos Ric}} = \max\{2,k+1\}$.

\item
  Another interesting condition is \emph{positive isotropic curvature}, of which S.\ Brendle and R.\ Schoen showed in \cite{brendle-schoen} that it is preserved under Ricci flow.
  For $d \geq 4 =: d_{\operatorname{pic}}$ we define the open convex cone
  \begin{align*}
    (\operatorname{pic})_d := \{ & R \in \AlgCurvOp[d] \mid\\
                               & \sect(R)(E_1, E_3)
                                 + \sect(R)(E_1, E_4)\\
                               & \qquad + \sect(R)(E_2, E_3)
                                 + \sect(R)(E_2, E_4)\\
                               & \qquad - 2\left<R(E_1 \wedge E_2)E_4 \wedge E_3\right> > 0\\
                               & \text{for any $\{E_1, \ldots, E_4\}$ orthonormal basis of a 4-plane in $\mathbb E^d$} \}.
  \end{align*}

\end{enumerate}
\noindent There are further examples for curvature conditions such as positive s-curvature, point-wise almost non-negative curvature (cf.\ \cite{hoelzel}) or positive $\Gamma_2$-curvature (cf.\ \cite{botvinnik-labbi}).
\end{example}

\noindent All of the above examples are given by open convex cones $C\subset \AlgCurvOp[d]$.

\subsection{Surgery stability}\label{sec:surgery-stable}

\noindent Let $\iota\colon N\embeds M$ be a $\codim0$-embedding and $h\in\calF(N)$. We define
\[\calF(M,\iota;h)\coloneqq \{g\in\calF(M)\colon \iota^*g = h\}.\]
This space can be thought of as the subspace of those metrics which have a fixed (standard) form on $N$. If $M$ has boundary $\partial M$ there is a restriction map $\res\colon\calF(M)\to\calR(\partial M)$ and for $h_\partial\in\calR(\partial M)$ we write
\[\calF(M)_{h_\partial} \coloneqq\res^{-1}(h_\partial)\]
Since the boundary of $N$ is collared, there is a collar of $\iota(N)\subset M$. If additionally $\iota(N)$ lies in the interior of $M$, then by prolonging the collar yields a homotopy equivalence $\calF(M,\iota;h) \simeq \calF(M\setminus\iota(N\setminus\partial N))_{\iota_*(h|_{\partial N})}$. We denote by $g_\circ^k \in \mathcal R(S^k)$ the round metric on the $k$-dimensional sphere.
\begin{defn}
	Let $c,d\in\bbN$. A Riemannian functor $\calF$ is called \emph{codimension $c$ surgery-stable on $d$-dimensional manifolds} if for every $k\le d-c$ there exists a metric $g^k\in\calF(S^{k}\times D^{d-k})_{g_\circ^{k} + g_\circ^{d-k-1}}$ such that for every embedding $\varphi\colon S^{k}\times D^{d-k}\embeds M$ into a $d$-manifold $M$ we have that
	\[
    \calF(M)\not=\emptyset\quad \Rightarrow\quad \calF(M,\varphi;g^k)\not=\emptyset.
  \]
	A Riemannian functor $\calF$ is called \emph{parametrized codimension $c$ surgery-stable on $d$-dimensional manifolds} if additionally the map
	\[\calF(M,\varphi;g^k)\embeds \calF(M)\]
	is a homotopy equivalence. We will abbreviate $\calF(M,\varphi;g^k)=\calF(M,\varphi)$. Usually there will be no chance of confusion and we will omit \enquote{on $d$-dimensional manifolds}.
\end{defn}

\noindent Let us first give an explanation for the wording \enquote{surgery-stability}. For this, let $c-1\le k\le d-c$ and let $\varphi\colon S^k\times D^{d-k}\embeds M$ be an embedding. We denote by $M_\varphi$ the manifold obtained by performing surgery on $M$ along $\varphi$. Let $\varphi^\op\colon D^{k+1}\times S^{d-k-1}\embeds M_\varphi$ be the obvious reversed surgery embedding. We now have the following observation: If $\calF$ is $\codim c$-surgery stable, then
\begin{align*}
	\calF(M)\not=\emptyset&\iff\calF(M,\varphi)\not=\emptyset\\
		&\iff \calF(\underbrace{M\setminus\im(\varphi)}_{= M_\varphi\setminus\im(\varphi^\op)})_{g_\circ^{k} + g_\circ^{d-k-1}} \not=\emptyset\\
		&\iff \calF(M_\varphi\setminus\im(\varphi^\op))_{g_\circ^{k} + g_\circ^{d-k-1}} \not=\emptyset\\
		&\iff\calF(M_\varphi,\varphi^\op)\not=\emptyset\quad\iff \calF(M_\varphi)\not=\emptyset
\end{align*}
If $\calF$ is parametrized $\codim c$-surgery stable, all of these spaces are homotopy equivalent:
\begin{cor}\label{cor:surgery-equivalence}
	Let $\calF$ be a parametrized codimension $c$ surgery stable Riemannian functor and let $\varphi\colon S^{k}\times D^{d-k}\embeds M$ be an embedding with $c-1\le k\le d-c$. Then we get a zigzag of weak homotopy equivalences
	\[\calS_{\calF,\varphi}\colon\calF(M)\overset{\simeq}{\hookleftarrow}\calF(M,\varphi)\congarrow\calF(M_\varphi,\varphi^\op)\overset{\simeq}\embeds \calF(M_\varphi).\]
\end{cor}

\noindent Let $M,N$ be $(c-2)$-connected $B\ort(d+1)\scpr{c-1}$-manifolds. If there is a $(c-2)$-connected $B\ort(d+1)\scpr{c-1}$-cobordism $X\colon M\leadsto N$ with a handle decomposition $H$ consisting only of handles of indices between $c$ and $(d-c+1)$, we get a well-defined homotopy class of a homotopy equivalence $\calS_{\calF,X,H}\colon\calF(M)\to\calF(N)$. We call the map $\calS_{\calF,X,H}$ the \emph{surgery map corresponding to $(X,H)$} and we note, that it depends on the choice decomposition. Any $B\ort(d+1)\scpr{c-1}$-cobordism can be turned into a $(c-2)$-connected one by performing surgery in the interior and then admits such a handle decomposition by the handle cancellation lemma from the proof of the $h$-cobordism theorem (cf.\ \cite{smale_structure}). Thus, if $M$ and $N$ are $B\ort(d+1)\scpr{c-1}$-cobordant, we have $\calF(M)\simeq\calF(N)$. In the case of positive scalar curvature this map has been studied by the first named author in \cite{actionofmcg}.

\begin{rem}
	\begin{enumerate}
		\item Note that with our definition $\codim c$-surgery-stability obviously implies $\codim c'$-surgery-stability for every $c'\ge c$.
		\item We do not explicitly assume the existence of a metric $\tilde g^k \in \calF(D^{k+1} \times S^{d-k-1})_{g_\circ^{k} + g_\circ^{d-k-1}}$ on the opposite surgery embedding in our definition of surgery stability, the reason being that it is not required in the proof of our main result. However, such a metric exists in all of the examples we know for surgery stability or if there is the symmetric lower bound on the index $k$ of the surgery embedding.
		\item Note that for all $c\le d$ we have that $\codim c$-surgery-stability of $\calF$ implies that $g_\circ+\dt^2\in\calF(S^{d-1}\times[0,1])_{g_\circ,g_\circ}$ by \pref{Proposition}{prop:cylinders}.
	\end{enumerate}
\end{rem}

\noindent For some of the constructions later on, we will need that fixing a metric on only one disk instead of $S^0\times D^d$ also gives a homotopy equivalence. This is guaranteed by the following proposition if $\calF$ is cellular.

\begin{prop}\label{prop:gluingdisk}
	Let $\calF$ be a parametrized codimension $d$ surgery stable Riemannian functor. Let $g^{0,1}\amalg g^{0,2} = g^0\in\calF(S^{0}\times D^{d})_{g_\circ,g_\circ}$. Then for any embedding $\iota\colon D^d\embeds M$ the inclusion
	\[\calF(M,\iota)\coloneqq\calF(M,\iota;g^{0,1})\embeds \calF(M)\]
	is a weak homotopy equivalence.
\end{prop}

\begin{proof}
	Without loss of generality we may assume that $\calF(M)\not=\emptyset$. Let $\varphi\colon S^0\times D^d\embeds M$ be an embedding that extends $\iota$ and consider the composition
	\[\calF(M,\varphi)\embeds \calF(M,\iota;g^{0,1}) \embeds \calF(M)\]
	which is a homotopy equivalence by parametrized surgery stability. Hence the second inclusion is surjective on all homotopy groups. For injectivity on homotopy groups let $\varphi\colon S^0\times D^d\embeds M\amalg M$ denote the disjoint union of $\iota$ with itself and consider the following diagram:
	\begin{center}
	\begin{tikzpicture}
		\node (0) at (0,1.2) {$\calF(M)$};
		\node (1) at (4,1.2) {$\calF(M\amalg M)$};
		\node (11) at (6.67,1.2) {$\calF(M)\times\calF(M)$};
		\node (2) at (0,0) {$\calF(M,\iota;g^{0,1})$};
		\node (3) at (4,0) {$\calF(M\amalg M,\varphi)$};
		\node (31) at (7.8,0) {$\calF(M,\iota;g^{0,1})\times\calF(M,\iota;g^{0,2})$};

		\draw[->] (0) to (1);
		\draw[->] (2) to (0);
		\draw[->] (2) to (3);
		\draw[double equal sign distance] (3) to (31);
		\draw[double equal sign distance] (1) to (11);
		\draw[->] (3) to node[right]{$\simeq$} (1);
	\end{tikzpicture}
	\end{center}
	The horizontal maps are inclusions into the product and hence injective on homotopy groups and it follows that the inclusion $\calF(M,\iota)\embeds\calF(M)$ is injective on homotopy groups.
\end{proof}

\begin{example}\label{ex:surgerystable}
  \begin{enumerate}
  \item It is well-known by the work of \cite{gromovlawson_classification} and \cite{schoenyau} that positive scalar curvature is codimension $3$ surgery-stable on $d$-manifolds in all dimensions $d \geq 3$.
    Chernysh showed in \cite{chernysh} that it is in fact parametrized codimension $3$ surgery-stable.
  \item A similar result is true for other open curvature conditions, which satisfy a condition specified by Hoelzel in \cite{hoelzel}.
    This includes curvature conditions such as positive $p$-curvature and $k$-positive Ricci curvature, which are codimension $p + 3$ (resp.\ $\max\{3, d-k+2\}$) surgery-stable on $d$-manifolds for $d \geq 3$.
    By work of the second named author \cite{kordass} these conditions are in fact parametrized surgery-stable with the same codimension restriction.
  \item The condition $\sec < 0$ gives rise to a Riemannian functor, which is codimension $2$ surgery stable on $2$-manifolds.
  \item The Riemannian functor, which assigns to a manifolds its metrics that are simultaneously conformally flat and have $\scal \geq 0$ is codimension $d$ surgery stable on $d$-manifolds (cf.\ \cite[Theorem 6.3]{hoelzel}).
  \end{enumerate}
\end{example}

\begin{rem}\label{rem:dk-ricci}
	Since we want the codimension restriction arising from surgery-stability to be independent of the dimension, we choose to replace $k$-positive Ricci curvature by  $(d-k)$-positive Ricci curvature, which is parametrized codimension $\max\{3,k+2\}$-surgery stable.
\end{rem}

\subsection{Fibrancy}\label{sec:fibrant}

In order to compare spaces of metrics on manifolds with different dimensions, we need the restriction map $\res\colon \calF(M)\to\calR(\partial M)$ to satisfy the properties from the following definition.

\begin{defn}\label{defn:fibrant}
  A Riemannian functor $\calF$ is called \emph{fibrant} if
  \begin{enumerate}
  \item $\res(\mathcal F(M)) \subset \mathcal F(\partial M)$ for all $M \in \mfd$ with $\partial M\not=\emptyset$  and
	\item the restriction map $\res\colon\calF(M)\to\calF(\partial M)$ is a Serre-fibration.
  \end{enumerate}
\end{defn}

\noindent The Riemannian functor given by positive scalar curvature is fibrant. This was shown utilizing the method we generalize here in \cite{ebertfrenck}.

\begin{prop}\label{prop:cylinders}
  A Riemannian functor $\mathcal F$ satisfies (1) in the above definition if and only if for every closed $N \in \mfd$ and every $g \in \mathcal R(N)$ with $g + \mathrm d t^2 \in \mathcal F(N \times [0,1])$ we have $g \in \mathcal F(N)$.
\end{prop}

\begin{proof}
  Let $N \in \mfd$ be closed with $g \in \calR(N)$ such that $g + \mathrm d t^2 \in \mathcal F(N \times [0,1])$. For $\res \colon \calF(N \times [0,1]) \to \calF(N \coprod N) = \calF(N) \times \calF(N)$ we get that $g = \pr_{\calF(N)}(\res(g + \mathrm d t^2))\in\calF(N)$.

  Now let $M \in \mfd$ with $\partial M \neq \emptyset$ and let $g \in \mathcal F(M)$. Since we assumed $M$ to be collared and the metric to be cylindrical in a neighborhood of the boundary, there is a $\codim 0$-embedding $c\colon [0,1]\times\partial M\embeds M$ such that $c^*g=\res(g) + \dt^2\in\calF(\partial M\times[0,1])$. By assumption, this implies that $\res(g)\in\calF(\partial M)$
\end{proof}

\noindent Let $\mathcal F$ be a Riemannian functor. For every closed manifold $N \in \mfd$ we have a continuous stabilization map
\begin{align*}
  \stab \colon \mathcal R(N) \to \mathcal R(N \times [0,1]),
  \quad
  g \mapsto g + \mathrm d t^2.
\end{align*}
The following is a criterion for curvature conditions for which $\calR_C$ is fibrant.

\begin{thm}\label{thm:fibrant}
	Let $C = \{C_d\}_{d \ge d_C}$ with $C_d \subset \AlgCurvOp[d]$ be a sequence of open curvature conditions. Let us assume that $\stab(\calR_C(N)) \subset \calR_C(N \times [0,1])$ for all closed $N \in \mfd$ and $\res(\calR_C(M))\subset\calR_C(\partial M)$ for all $M \in \mfd$ with $\partial M \neq \emptyset$. Then $\calR_C$ is fibrant.
\end{thm}

\noindent Before diving into the proof, let us give the consequences most important to us.

\begin{prop}\label{prop:fibrant}
 Both positive $p$-curvature and $(d-k)$-positive Ricci curvature are fibrant.
\end{prop}

\begin{proof}
	It remains to show that $g$ has positive $p$-curvature (resp. $(d-k)$-positive Ricci curvature) if and only if $g+\dt^2$ has positive $p$-curvature (resp. $(d+1-k)$-positive Ricci curvature).

	If $g$ has positive $p$-curvature, then the $p$-curvature of $g+\dt^2$ is positive by the computation in \pref{Lemma}{lem:pstable}. Now let $g+\dt^2$ have positive $p$-curvature and let $P\subset T_xM$ be a $p$-dimensional subspace. Then there is an orthonormal basis $(\partial_t,\dots, E_{d+1-p})$ of $P^\perp$ in $T_{(x,t)}M\times[0,1]$ and we can compute
	\[s_{p,g}(P) = s_{p,g+\dt^2}(P) - \sum_{i=2}^{d+1-k}\underbrace{\sec(\partial_t,E_i)}_{=0}>0.\]
	\noindent Concerning $(d-k)$-positive Ricci we note that the eigenvalue of the Ricci curvature corresponding to $\partial_t$ equals $0$, the first sum of the first $(d+1-k)$ eigenvalues of $\ric(g+\dt^2)$ is positive if and only if the sum of the first $(d-k)$-eigenvalues of $\ric(g)$ is positive.
\end{proof}

\noindent We have the following observation:

\begin{lem}\label{lem:gajer}
	Under the assumptions in \pref{Theorem}{thm:fibrant}, $\calR_C(M)$ satisfies the following: For every closed manifold $N^{d-1} \in \mfd$ we have:
  \begin{enumerate}
    \setcounter{enumi}{\value{temporaryResumeCounter}}
  \item For every smooth path of Riemannian metrics $\{g_r\}_{r \in [0,1]} \subset \mathcal R(N)$ with $g_r + \mathrm d t^2 \in \mathcal F(N \times [0,1])$ for all $r \in [0,1]$, there exists a $0 < \Lambda \leq 1$ such that  we have $g_{f(t)} + \mathrm d t^2 \in \mathcal F(N \times [0,1])$ for every function $f \colon \mathbb R \to [0,1]$ that is constant near $0$ and $1$ and satisfies $|f'|, |f''| \leq \Lambda$
  \item Additionally, $\Lambda$ can be chosen depending continuously on the family $\{g_r\}$.
  \end{enumerate}
\end{lem}

\begin{proof}
	  We obtain (1) immediately from a computation similar to \cite[p.184]{gajer} (cf.\ \pref{Lemma}{lem:computation-gajer}), which yields the following correspondence between curvature tensors:
  \begin{align*}
    R_{(N \times \bbR, g_{f(t)} + \mathrm{d} t^2)}|_{(x,t_0)} & = R_{(N \times \bbR, g_{f(t_0)} + \mathrm{d} t^2)}\\
                                                              & \qquad  + O(|f'|) E_1 + O(|f'|^2) E_2 + O(|f''|) E_3,\nonumber
  \end{align*}
  where $E_1, E_2, E_3$ only depend on the path $\{g_r\}_{r \in [0,1]}$ and its derivatives in $r$-direction.
  Since $C$ is an open subset in $\mathcal C_{\mathrm B}(\mathbb E^n)$, we find $\Lambda$ accordingly. This also reveals that $\Lambda$ can be chosen continuously and thus implies (2).
\end{proof}

\begin{remark}\leavevmode
	\begin{enumerate}
		\item The proof of \pref{Theorem}{thm:fibrant} indeed shows the following:
	\end{enumerate} 
	\vspace{3pt}
		\begin{changemargin}{1cm}{1cm}
		{\noindent If $\calF$ is an open Riemannian functor satisfying the two properties from \pref{Lemma}{lem:gajer} together with the property that $\stab(\calF(N)) \subset \calF(N \times [0,1])$ for all closed $N \in \mfd$ and $\res(\calF(M))\subset\calF(\partial M)$ for all $M \in \mfd$ with $\partial M \neq \emptyset$, then $\calF$ is fibrant. }
\end{changemargin}
	\vspace{3pt}
	\begin{enumerate}
		\item[]	However, since the examples we are interested in are all given by curvature conditions, we decided to simplify the statement our criterion by only considering subspaces given by curvature conditions.
		\item[(2)] Given a path $\{g_r\}_{r \in [0,1]}$ with $g_r + \mathrm d t^2 \in \mathcal F(N \times [0,1])$ for every $r \in [0,1]$, (1) from \pref{Lemma}{lem:gajer} implies the existence of a metric $G\in\calF(N\times[0,1])_{g_0,g_1}$
	\end{enumerate}
\end{remark}

\noindent Let us now turn to the proof of \pref{Theorem}{thm:fibrant}. The following lemma and its proof are adaptations from \cite[Lemma 5.1]{ebertfrenck} to a more general setting. It constructs a family of paths from a path of metrics, which stops at any particular point.

\begin{lem}\label{lem:path-to-family}
  Let $\mathcal F$ be an open Riemannian functor.
  Let $N^{d-1}$ be a closed manifold, $P$ be a compact topological space and let $G \colon P \times [0,1] \to \mathcal F(N)$ be a continuous map.
  Then there exists a continuous map
  \begin{align*}
    C \colon P \times [0,1]^2 \to \mathcal F(N),
    \quad
    (p,s,t) \mapsto C(p,s,t).
  \end{align*}
  with the properties
  \begin{enumerate}
  \item $\{C(p,s,t)\}_{t \in [0,1]}$ is a smooth path of metrics for every $(p,s) \in P \times [0,1]$,
  \item $C(p,0,t) = G(p,0)$ for all $(p,t) \in P \times [0,1]$,
  \item $C(p,s,0) = G(p,0)$ for all $(p,s) \in P \times [0,1]$,
  \item $C(p,s,1) = G(p,s)$ for all $(p,s) \in P \times [0,1]$.
  \end{enumerate}
  If, additionally, $\mathcal F = \calR_C$ satisfies the assumptions of \pref{Theorem}{thm:fibrant}, then there exists $0 < \Lambda \leq 1$ such that for every function $f \colon \mathbb R \to [0,1]$ with $|f'|, |f''| \leq \Lambda$ we have $C(p, s, f(t)) + \mathrm d t^2 \in \mathcal F(N \times [0,1])$.
\end{lem}

\begin{proof}
  First, note that since $\calF(N)$ is open, we may without loss of generality assume that for every $p \in P$ the path $\{G(p, r)\}_{r \in [0,1]} \subset \mathcal F(N)$ is smooth. As in \cite[Lemma 5.1]{ebertfrenck}, we let $U_{ni} := (\frac{i-1}{n},\frac{i+1}{n}) \cap [0,1]$, define the open cover $\mathcal U_n = \{ U_{ni} \mid i = 0, \ldots, n \}$ of $[0,1]$ and choose a subordinate smooth partition of unity $\{\lambda_{ni} \mid i = 0, \ldots, n \}$ to define:
  \begin{align*}
    C_n \colon P \times [0,1]^2 \to \mathcal R(N),
    \quad (p,s,t) \mapsto \sum_{i=0}^n G(p, \frac{s \cdot i}{n})\lambda_{ni}(t).
  \end{align*}
  This converges uniformly to $G$ in the sense that $\lim_{n \to \infty}C_n(p,s,t) = G(p,s \cdot t)$.
  Again using that $\calF(N)$ is open, we conclude there exists a sufficiently large $n$ such that $\operatorname{Im}(C_n) \subset \calF(N)$.
  We then let $C := C_n$.

  If $\calF$ satisfies the assumptions of \pref{Theorem}{thm:fibrant}, then $\stab(\calF(N)) \subset \calF(N \times [0,1])$ and therefore $\{C(p,s,t_0) + \mathrm{d}t^2\}_{t_0 \in [0,1]}$ is a smooth path within $\calF(N \times [0,1])$ for every $(p,s) \in P \times [0,1]$. By (1) in \pref{Lemma}{lem:gajer} we find a $0 < \Lambda_{(p,s)} \leq 1$ such that for every function $f \colon \mathbb R \to [0,1]$ with $|f'|, |f''| \leq \Lambda_{(p,s)}$ we have $C(p, s, f(t)) + \mathrm d t^2 \in \mathcal F(N \times [0,1])$. Finally by (2) in \pref{Lemma}{lem:gajer}, $\Lambda_{(p,s)}$ depends on $(p,s)$ continuously and thus we choose $\Lambda := \min \{\Lambda_{(p,s)}\}$.
\end{proof}

\begin{proof}[Proof of \pref{Theorem}{thm:fibrant}]
  To prove the statement, it suffices to find a solution to the following lifting problem:
  \begin{equation*}
    \begin{tikzcd}
      D \times \{0\} \ar{r}{h} \ar[hook]{d} & \mathcal F(M) \ar{d}{\res} \\
      D \times [0,1] \ar{r} \ar{r}{G} \ar[dashed]{ru} & \mathcal F(\partial M)
    \end{tikzcd}
  \end{equation*}
  where $D$ is a disc. We choose $\delta > 0$ such that $h(D \times \{0\}) \subset \mathcal F(M)$ is of product form on the collar of length $2 \delta$. Since $\mathcal F$ is open, $G$ is homotopic relative to $G|_{D \times \{0,1\}}$ to a map $\tilde G$ with $\{\tilde G(p,t)\}_{t \in [0,1]}$ a smooth path of metrics for every $p \in P$.
  We replace $G$ by $\tilde G$.

  Now apply Lemma \ref{lem:path-to-family} to $G$ to obtain a map $C \colon D \times [0,1]^2 \to \mathcal F(\partial M)$ and $0 < \Lambda \leq 1$ accordingly. Choose a smooth function $f \colon \mathbb R \to [0,1]$ such that $|f'|, |f''| \leq \Lambda$ and $f|_{(-\infty, 0]} \equiv 0$, $f|_{[b, \infty)} \equiv 1$ for $b > \delta > 0$ sufficiently large. Using the collar of $M$, we define $M' = M \setminus (\partial M \times [0,\delta])$ and thus we can write $M = M' \cup_{\partial M} (\partial M \times [0,\delta])$.
  Now choose a monotone diffeomorphism $\phi \colon [0,\delta] \to [0,b]$ with $\phi' = 1$ near $0$ and $\delta$. Thus a candidate for a lift is given by
  \begin{align*}
    \hat G \colon D \times [0,1] & \to \mathcal F(M' \cup_{\partial M} (\partial M \times [0,\delta])) = \mathcal F(M),\\
    (p, s) & \mapsto
    h(p)|_{M'} \cup (\id_{\partial M} \times \phi)^{*}(C(p,s,f(t)) + \mathrm dt^2).
  \end{align*}
  This is well-defined, since along the gluing, we have (cf.\ (2) in Lemma \ref{lem:path-to-family}) for all $p \in D, s \in [0,1]$
  \begin{align*}
    \hat G(p,s)|_{\partial M \times \{0\}} = C(p,s,0) = G(p,0) = \res(h(p)).
  \end{align*}
  Moreover, by construction of $C$ (cf.\ (3) in Lemma \ref{lem:path-to-family}) we have for $p \in D$:
  \begin{align*}
    \hat G(p, 0)
    = h(p)|_{M'} \cup (\id_{\partial M} \times \phi)^{*}(G(p,0) + \mathrm dt^2)
    = h(p)
  \end{align*}
  and (cf.\ (4) in Lemma \ref{lem:path-to-family})
  \begin{align*}
    \res(\hat G(p, s))
    = C(p, s, f(b))
    = C(p, s, 1)
    = G(p,s)
  \end{align*}
  for $p \in D$, $s \in [0,1]$.
  Hence, $\hat G$ makes the diagram commute and is indeed a lift.
\end{proof}

\subsection{Statement of main result, general version}\label{sec:thmmain}
	Having introduced all the necessary notions, we can now state the general version of our main result.

\begin{thm}\label{thm:maingeneral}
	Let $n\ge c\ge3$ and let $\calF$ be a cellular, parametrized codimension $c$-surgery stable Riemannian functor that implies positive scalar curvature. Let $W$ be a $\Spin$-manifold of dimension $d=2n$. Let $h\in\calR^+(\partial W)$ and $g\in \calF(W)_h$. Then for all $k\ge0$ such that $d+k+1\equiv 0(4)$ the composition
	\[\pi_{k}(\calF(W)_h)\too\pi_{k}(\calR_\psc(W)_h)\overset{\inddiff_{g}}\too\ko^{d+k+1}(\pt)\cong\bbZ\]
	is nontrivial.
	If additionally $\calF$ is fibrant, this holds for all manifolds of dimension $d\ge2c$.
\end{thm}

\noindent \pref{Theorem}{thm:mainpcurv} and \pref{Theorem}{thm:maindkric} now follow from the above theorem by  \pref{Example}{ex:surgerystable} and \pref{Proposition}{prop:fibrant}. Note that the long list of adjectives in front of \enquote{Riemannian functor} does not imply lack of examples but rather is due to the fact that there are many examples and the aim to extract necessary assumptions out of these.

\subsection{Stable metrics}

The following Lemma states the existence of stable metrics (cf.\ \cite{erw_psc2}) in a special case. Let $c\ge3$ and let $\calF$ be a parametrized codimension $c$ surgery stable Riemannian functor.

\begin{lem}[{\cite[Theorem 2.6]{berw}}]\label{lem:stable-metrics}
		Let $d\ge 2c-1$ and let $V^d\colon S^{d-1} \leadsto S^{d-1}$ be a $(c-2)$-connected, $BO(d)\langle c-1\rangle$-cobordism. Also, assume that $V$ is $BO(d)\langle c-1\rangle$-cobordant to $S^{d-1}\times [0,1]$ relative to the boundary. Then there exists a metric $g\in\calF(V)_{g_\circ,g_\circ}$ with the following property: If $W:S^{d-1}\leadsto S^{d-1}$ is cobordism and $h\in\calR(S^{d-1})$ is a boundary condition then the two gluing maps
		\begin{align*}
			\mu(\_,g)\colon\calF(W)_{h,g_\circ}&\too \calF(W\cup V)_{h,g_\circ}\\
			\mu(g,\_)\colon\calF(W)_{g_\circ,h}&\too \calF(V\cup W)_{g_\circ, h}
		\end{align*}
		are homotopy equivalences.
\end{lem}

\begin{defn}\label{def:stable-metrics}
A metric $g$ as in this Lemma is called an \emph{$\calF$-stable metric}.
\end{defn}
\begin{proof}[Proof of \pref{Lemma}{lem:stable-metrics}]
	By assumption, there exists a relative $BO(d)\langle c-1\rangle$-cobordism $X\colon V\leadsto S^{d-1}\times [0,1]$ and by performing surgery on the interior of $X$ we may assume $X$ has no handles of indices $0,\dots, c-1, d+1-c+1,\dots, d+1$. So $S^{d-1}\times [0,1]$ is obtained from $V$ by a sequences of surgeries in the interior with these indices. For $i=1,\dots,l$ let $\varphi_i\colon S^{k_i}\times D^{d-k_i}\embeds V_i$ with $V_0=V$ and $V_{i+1}\coloneqq (V_{i})_{\varphi_i}$ be the corresponding sequence of surgery embeddings with $k_i\in\{c,\dots, d+1-c\}$. Let $g\in\calF(V)_{g_\circ,g_\circ}$ such that $g_\circ + \mathrm{d} t^2\in[\sfphi{l}\circ\dots\circ\sfphi{1}(g)]\in\pi_0(\calF(S^{d-1}\times[0,1])_{g_\circ,g_\circ})$ which is possible since the maps $\sfphi i$ are homotopy equivalences. Now $\mu(g_\circ+\dt^2,\_)$ (resp. $\mu(\_,g_\circ+\dt^2)$) is a homotopy equivalence and hence so is $\mu(g,\_)$ (resp. $\mu(\_,g)$).
%
%
\end{proof}

\subsection{Madsen-Tillmann-Weiss spectra}\label{sec:mtw}

We briefly recall the definition of structured Madsen--Tillmann--Weiss spectra\footnote{See \cite{gmtw} or \cite{gollinger} for a more detailed introduction.}. Let $B\ort(d)$ denote the classifying space of $\rk d$-vector bundles and let $U_d\to B\ort(d)$ be the universal vector bundle. The orthogonal complement of $U_d$ which is a virtual vector bundle is denote by $U_d^\perp$. Let
\[B_n(d)\coloneqq B\ort(d)\scpr{n}\overset{\theta_n(d)}\too B\ort(d)\]
be the $n$-connected cover of $B\ort(d)$ with $B_n\coloneqq\mathrm{colim}_dB_n(d)$. We define the spectrum $\mt\theta_n(d)$ as the Thom spectrum of $U_d^\perp$, \ie
\[\mt\theta_n(d)\coloneqq \Th(\theta_n(d)^*U_d^\perp)\]
Note that for $d\ge3$ we have $B_2(d)=B\Spin(d)$ and for $n\ge3$ we get a map
\[\mt\theta_n(d) \too \mt\Spin(d)\]
By \cite[p.796]{berw} there is a spectrum map $\lambda_{-d}\colon \mt\Spin(d)\to \Sigma^{-d}KO$ and we have the composition $\mt\theta_n(d) \too \mt\Spin(d)\to\Sigma^{-d}KO$. We get the following induced maps on rational homotopy groups.
\begin{align*}
	\pi_k(\mtt_n(d))\otimes\bbQ \to \pi_k(\mt\Spin(d))\otimes\bbQ \to &\ko^{-d-k}(\pt)\otimes\bbQ\\
		&\qquad \cong \begin{cases}\bbQ &\text{ if } d+k\equiv 0(4)\\0&\text{ else} \end{cases}
\end{align*}
where the first map is induced by the inclusion. By the Pontryagin--Thom construction the group $\pi_k(\mtt_n(d))$ is isomorphic to the cobordism group of triples $[M,V,\phi]$ where $M$ is a closed $(k+d)$-manifold, $V\to M$ is a $\rk d$ vector bundle with a $\theta_n$ structure and $\phi\colon V\oplus\underline\bbR^{k} \cong TM$ is a stable isomorphism of vector bundles (cf.\ \cite[Theorem 5.1]{berw} or \cite[Proposition 1.2.3]{gollinger}. The map $\pi_k(\mtt_n(d))\otimes\bbQ \to \pi_k(\mt\Spin(d))\otimes\bbQ$ is the forgetful map. Note that in the case $d+k\equiv0(4)$ the triple $[M,V,\phi]$ gets mapped to the $\hat\calA$-genus $\hat\calA(M)$ of $M$ under the above composition by \cite[p. 817]{berw}. We will denote the above composition by $\hat \calA\otimes\bbQ$.

\begin{thm}\label{thm:rational-mtw}
	$\hat\calA\otimes\bbQ\colon\pi_k(\mtt_n(d))\otimes\bbQ\to\ko^{-d-k}(\pt)\otimes\bbQ$ is surjective, provided $d>n+1$.
\end{thm}

\noindent For the proof we need the following lemma.

\begin{lem}\label{lem:pontryagin}
	For  $n<d$ we have $H^*(B_n(d),\bbQ)\cong\bbQ[p_{\floor{n/4} +1},\dots,p_{\floor{d/2}}]$.
\end{lem}

\begin{proof}
	The proof is by induction over  $n$ and all  cohomology here is with rational coefficients. For $n=1$ we have that $B_n(d) =  B\SO(d)$ and $H^*(B\SO(d))\cong\bbQ[p_{1},\dots,p_{\floor{d/2}}]$ is well known (cf.\ \cite[Lemma 2.4]{brown_cohomology}). For $n\ge2$ we have a fibration
	\[B_n(d)\too B_{n-1}(d) \too K(\pi_n(B\SO(d)),n).\]
	Note that because of $n<d$ we have $\pi_n(B\SO(d))$ is either $\bbZ$ (for $n\equiv0(4)$), $\bbZ/2$ (for $n\equiv1,2(8)$) or $0$ and hence it suffices to consider the case that $n=4m$ because in the other cases the map $B_n(d)\to B_{n-1}(d)$ induces an isomorphism in rational cohomology. The Serre spectral sequence has the form
	\[E_{p,q}^2 = H^{p}(K(\bbZ,n))\otimes H^q(B_n(d))\Rightarrow H^{p+q}(B_{n-1}(d)).\]
	The cohomology of  $K(\bbZ,n)$ is given by $H^{*}(K(\bbZ,n)) \cong\bbQ[\alpha]$ for $\alpha$ in degree $n$. Furthermore, $H^{p}(B_n(d))=0$ unless $p\equiv0(4)$. Hence all differentials vanish, the spectral sequence collapses on the $E_2$-page and we have
	\[\bigoplus_{p+q=k} H^p(K(\bbZ,n))\otimes H^q(B_n(d)) \congarrow H^k(B_{n-1}(d)).\]
	Since $H^{n}(B_n(d))=0$, the preimage of $p_{m}$ is the class $\alpha$ generating $H^*(K(\bbZ,n))$ and therefore $H^*(B_n(d),\bbQ)\cong\bbQ[p_{m+1},\dots,p_{\floor{d/2}}]$.
\end{proof}

\begin{cor}\label{cor:mtt-bordism}
	The bordism group $\pi_k(\mtt_n(d))\otimes\bbQ$ consists of the classes in $\Omega_{d+k}^{\theta_n}\otimes\bbQ$ which do not have nontrivial Pontryagin classes of degree greater than $\floor{d/2}$.
\end{cor}

\begin{proof}
	Since the sphere spectrum is rationally an $H\bbQ$-spectrum by Serre's finiteness theorem, the rational Hurewicz-homomorphism of spectra $\pi_k(\mtt_n(d))\otimes\bbQ\to H_k(\mtt_n(d);\bbQ)$ is an isomorphism. Composing with the Thom-isomorphism we get an isomorphism $\pi_k(\mtt_n(d))\otimes\bbQ\to H_{k+d}(B_n(d),\bbQ)$. The claim follows from \pref{Lemma}{lem:pontryagin} by considering the natural map $H_{k+d}(B_n(d),\bbQ)\embeds H_{k+d}(B_n,\bbQ)\cong \Omega_{d+k}^{\theta_n}\otimes\bbQ$ (cf.\ \cite[Theorem 2.1]{krecklueck} for the last isomorphism).
\end{proof}

\begin{proof}[{Proof of \pref{Theorem}{thm:rational-mtw}}]
	Again we restrict to the case $d+k\equiv0(4)$. By the isomorphism $H_*(B_n,\bbQ)\cong \Omega_*^{\theta_n}\otimes\bbQ$, there are nontrivial classes $M_l\in\Omega_{4l}^{\theta_n}\otimes\bbQ$ for $l\in\{\floor{n/4} +1,\dots,2\floor{n/4} +1\}$. Note that by our assumption $2\floor{n/4} +1\le\floor{d/2}$. Since all Pontryagin classes of $M_l$ until $p_{\floor{n/4}+1}$ vanish, the only nontrivial  Pontryagin number of $M_l$ is $\scpr{p_l(M_l),[M_l]}$. By \cite[Theorem 4]{bergberg} this number is a multiple of the $\hat\calA$-genus. By the euclidian algorithm there exists a $q\in\bbR$ and an $r\in\{\floor{n/4} +1,\dots,2\floor{n/4} +1\}$ such that  $\frac{d+k}4 = q\cdot(\floor{n/4} + 1) + r$ and hence $M\coloneqq (M_{\floor{n/4} + 1})^q \times M_r$ has only Pontryagin classes of degree smaller than $\floor{d/2}$ and hence is an element of $\pi_k(\mtt_n(d))\otimes\bbQ$ by \pref{Corollary}{cor:mtt-bordism} with non-vanishing $\hat\calA$-genus which proves the theorem.
\end{proof}

\subsection{Index theoretic ingredients}\label{sec:index-theory}

\noindent This is mainly a recollection of index-theoretic arguments involved in the proof of our main result.  Even though this is just a recollection from \cite{berw}, we decided to keep it in here to give some context. There is no claim of originality for this entire section.

\subsubsection{$\ko$-theory}

Let us start by recalling the model for $\ko$-theory that was used in \cite[Chapter 3]{berw}, for a more detailed discussion see loc.cit.. Let $X$ be a space $H\to X$ be a Hilbert bundle with separable fibers. An \emph{operator family} is a fiber preserving and fiber-wise linear continuous map $H_0\to H_1$ of Hilbert bundles $H_0$ and $H_1$. It is determined by a family $(F_x)_{x\in X}$ of bounded  operators $F_x\colon(H_0)_x\to (H_1)_x$. $F$ is called \emph{adjointable} if $(F_x)^*_{x\in X}$ is an operator family and we denote the algebra of adjointable operators by $\lin_X(H)$. The $*$-ideal of compact operators on $X$ is denoted by $\kom_X(H)$. We call an adjointable operator family $F$ a \emph{Fredholm family} if there exists a $K\in\kom_X(H)$ such that $F+K$ is invertible.

\begin{defn}
	Let $V\to X$ be a Riemannian vector bundle and let $\tau\colon V\to V$ be a self-adjoint involution. A \emph{$\cl(V^\tau)$-Hilbert bundle} is a triple $(H,\iota,c)$ where $H\to X$ is a Hilbert bundle, $\iota\colon H\to H$ is a self-adjoint involution and $c=(c_x)_{x\in X}$ is a collection of maps $c_x\colon V_x\to \lin(H_x)$ such that
	\begin{enumerate}
		\item $c_x(v)\iota + \iota c_x(v) = 0$
		\item $c_x(V)^* = c_x(\tau v)$
		\item $c_x(v) \cdot c_x(v') + c_x(v') \cdot c_x(v) = -2\scpr{v,\tau v}$
		\item If $s\in\Gamma(X,V)$ is a continuous section, then $c_x(s(x))$ is an operator family
	\end{enumerate}
	We will omit $x$ in $c_x(v)$ when there is no chance of confusion.
\end{defn}

\noindent The opposite $\cl(V^\tau)$-Hilbert bundle is given by $(H,-\iota,-c)$. A \emph{$\cl(V^\tau)$-Hilbert bundle} with $V=V^+ \oplus V^-$ and $\tau(v_1,v_2) = (v_1,-v_2)$. It will also be called a \emph{$\cl(V^+\oplus V^-)$-Hilbert bundle} and if $V^+= \underline{\bbR^p}$ and $V^-=\underline\bbR^q$ we will abbreviate this by $\cl^{p,q}$. A \emph{$\cl(V^+\oplus V^-)$-module} is a finite-dimensional $\cl(V^+\oplus V^-)$-Hilbert bundle and a \emph{$\cl^{p,q}$-Fredholm family} is a Fredholm family on a $\cl^{p,q}$-Hilbert bundle that is $\cl^{p,q}$-linear and anti-commutes with the grading, \ie $Fc(v) = c(v)F$ and $F\iota = -\iota F$.

\begin{defn}\label{def:concordance}
	Let $(X,Y)$ be a space pair. A \emph{$(p,q)$-cycle on $X$} is a tuple $(H,\iota,c,F)$ where $(H,\iota,c)$ is a $\cl^{p,q}$-Hilbert bundle and $F$ is a $\cl^{p,q}$-Fredholm family. A \emph{relative $(p,q)$-cycle} is a $(p,q)$-cycle on $X$ such that $F$ is invertible over $Y$. A \emph{concordance between $(H_0,\iota_0,c_0,F_0)$ and $(H_1,\iota_1,c_1,F_1)$} is a relative $(p,q)$-cycle $(H,\iota,c,F)$ on $(X,Y)\times[0,1]$ such that $(H,\iota,c,F)|_{X\times\{i\}}=(H_i,\iota_i,c_i,F_i)$. A $(p,q)$-cycle is called \emph{acyclic} if $F$ is invertible.
\end{defn}

\noindent We will sometimes abbreviate $(H,\iota,c,F)$ by $(H,F)$ or $x\mapsto (H_x,F_x)$.

\begin{defn}
	For a pair $(X,Y)$ of a paracompact space $X$ and a closed subspace $Y$ we define
	\[F^{p,q}(X,Y) \coloneqq \frac{\{\text{concordance classes of relative }(p,q)\text{-cycles}\}}{\{\text{concordance classes of acyclic ones}\}}.\]
	This is an abelian group via direct sum and the inverse of $[H,\iota,c,F]$ is given by $[H, -\iota,-c,F] = [H,\iota,-c,-F]$.
\end{defn}
\noindent For  space pairs $(A,B)$ and $(X,Y)$ let $(A,B)\times (X,Y)\coloneqq (A\times X, A\times Y\cup B\times X)$. Let $\bbI\coloneqq[-1,1]$. We define $\Omega F^{p,q}(X,Y)\coloneqq F^{p,q}((X,Y)\times(\bbI,\partial\bbI))$. There is an isomorphism of abelian groups
\begin{align*}
	F^{p,q}(X,Y)&\too\Omega F^{p-1,q}(X,Y)\\
		(H,F)&\mapsto \bigl( (x,s)\mapsto (H_x,F_x+s\iota_xc(e_1)_x)\bigr)
\end{align*}

A $\cl^{p,q}$-Hilbert space is called \emph{ample} if it contains every finite dimensional irreducible $\cl^{p,q}$-Hilbert space with infinite multiplicity. We fix an ample $\cl^{p,q}$-Hilbert space $U$ and we define $\fred^{p,q}$ to be the space of all $\cl^{p,q}$-Fredholm  operators on $U$ with the norm topology and $G^{p,q}$ the (contractible) subspace of invertible ones. We have the following theorem.
\begin{thm}[{\cite[Theorem 3.3 and below]{berw}}]
	Let $(X,Y)$ be a CW-pair. Then the following holds\begin{enumerate}
		\item There is an isomorphism $\ko^{q-p}(X,Y) \congarrow F^{p,q}(X,Y)$.
		\item Every class $b\in F^{p,q}(X,Y)$ corresponds to a unique homotopy class of a map $(X,Y)\to(\Omega^{\infty+p-q}\ko,*)$ which we call the \emph{homotopy-theoretic realization of $b$}
	\end{enumerate}
\end{thm}

\subsubsection{Dirac operators}
Let $W^d$ be a Riemannian $\Spin$-manifold. So, we have a $\Spin(d)$-principal bundle $P\to W$ and an isometry $\eta\colon P\times_{\Spin(d)}\bbR^d\to TW$. The \emph{spinor-bundle} $\frakS_W$ of $W$ is the associated fiber-wise irreducible real $\cl(TW^+\oplus\bbR^{0,d})$-module. The Levi-Civita connection on $W$ induces a canonical connection $\nabla$ on $\frakS_W$. The \emph{Dirac operator} $\frakD$ is given by
\[\frakD\colon\Gamma(W,\frakS_W)\overset{\nabla}\too\Gamma(W;TW\otimes\frakS_W)\overset{c}\too\Gamma(W;\frakS_W).\]
$\frakD$ is a linear, formally self-adjoint, elliptic differential operator of order 1 and anti-commutes with the grading and the Clifford multiplication of $\bbR^{0,d}$. Hence, after changing the $\cl^{0,d}$-multiplication to a $\cl^{d,0}$-multiplication via replacing $c(v)$ by $\iota c(v)$, the Dirac operator $\frakD$ becomes $\cl^{d,0}$-linear. The relevance of the Dirac operator to positive scalar curvature geometry originates from the Schrödinger-Lichnerowicz formula:
\[\frakD^2 = \nabla^*\nabla + \frac{\scal}4,\]
which forces the Dirac operator to be invertible if the scalar curvature is positive. Now, let $X$ be a paracompact Hausdorff space and $\pi\colon E\to X$ a fiber bundle with possibly non-compact $E_x\coloneqq\pi^{-1}(\{x\})$ of dimension $d$ such that the vertical tangent bundle $T_{(v)}E$ admits a $\Spin$-structure. A fiber-wise Riemannian metric $g_x$ gives rise to a Spinor-bundle $\frakS_E$, a $\cl(T_{(v)}E^+\oplus\bbR^{0,d})$-module that restricts to the Spinor-bundle $\frakS_x\to E_x$ with Dirac operator $\frakD_x$ in each fiber. If the fibers are compact with boundary diffeomorphic to $N$ and the boundary bundle is trivial as a $\Spin$-bundle, \ie $\partial E = X\times \partial N$, we can consider the \emph{elongation} of $(E,g)$. This is defined to be the bundle $\hat E\coloneqq E\cup_{\partial E}(X\times[0,\infty)\times N)$ with the metric $(dt^2 + g_x)$ on the added cylinders.

\begin{defn}[{\cite[Definition 3.4]{berw}}]
Let $t\colon E\to \bbR$ be fiber-wise smooth such that $(\pi,t)\colon E\to X\times\bbR$ is proper. Let $a_0<a_1\colon X\to[-\infty,\infty]$ be continuous functions. We define
\[X\times(a_0,a_1)\coloneqq \{(x,s)\in X\times\bbR \mid  a_0(x)<s<a_1(x)\}\]
and $E_{(a_0,a_1)}\coloneqq (\pi,t)^{-1}(X\times(a_0,a_1))$. We say the bundle $E$ is \emph{cylindrical over $(a_0,a_1)$} if there exists a bundle isomorphism $E_{(a_0,a_1)}\cong (X\times \bbR\times M)_{(a_0,a_1)}$ for some $(d-1)$-manifold $M$. $E$ is said to have \emph{cylindrical ends}, if $E$ is cylindrical over $(-\infty,a_-)$ and $(a_+,\infty)$ for some functions $a_-,a_+\colon X\to\bbR$. If $E$ has cylindrical ends and there is a fiber-wise Riemannian metric $g=(g_x)_{x\in X}$ that is cylindrical over the ends, we say that $(g_x)$ has \emph{positive scalar curvature at infinity} if there exists a function $\epsilon\colon X \to (0,\infty)$ such that on the ends of $E_x$ the metric $g_x$ has scalar curvature $\ge\epsilon(x)$.
\end{defn}

\noindent Let $L^2(E,\frakS_E)_x$ denote the Hilbert space of $L^2$-sections of the spinor bundle $\frakS_x\to E_x$. These assemble to a $\cl^{d,0}$-Hilbert bundle over $X$. The Dirac operator is densely defined symmetric unbounded operator on $L^2(E,\frakS_E)_x$ and its closure is self-adjoint. Applying the functional calculus for $f(x)=\frac{x}{\sqrt{1+x^2}}$ we get the \emph{bounded transform}
\[F_x\coloneqq \frac{\frakD_x}{\sqrt{1+\frakD_x^2}}.\]
If $g$ has positive scalar curvature at infinity this is a bounded $\cl^{d,0}$-Fredholm operator. The collection $(F_x)$ is a $\cl^{d,0}$-Fredholm family over $X$. We define $\dirac(E,g)$ to be the $(d,0)$-cycle given by $x\mapsto(L^2(E,\frakS_E)_x,F_x)$ and if we assume that $\frakD_y$ is invertible for all $y\in Y$ we obtain a class
\begin{align*}
	\ind(E,g)\coloneqq[\dirac(E,g)]\in\ko^{-d}(X,Y).
\end{align*}
We have the following Lemma.
\begin{lem}[{\cite[Lemma 3.7]{berw}}]
	Let $\pi\colon E\to X$ be a $\Spin$-manifold bundle with cylindrical ends and let $g_0, g_1$ be fiber-wise metrics with psc at infinity and agree on the ends. Let $Y\subset X$ and assume that $g_0$ and $g_1$ agree and have invertible Dirac operators over $Y$. Then
	\[\ind(E,g_0) = \ind(E,g_1)\in \ko^{-d}(X,Y).\]
\end{lem}

\noindent In particular, if $g$ is a fiber-wise metric only defined over the ends we still get a well defined class $\ind(E,g)$ and if $E$ is closed, then $\ind(E,g)$ does not depend on $g$. Also note that for the bundle $E^\op$ with the opposite $\Spin$-structure we have $\ind(E^\op,g) = - \ind(E,g)$.
\vspace{0.4em}
{\center{\emph{From now on let $\calF$ be a Riemannian functor that implies positive scalar curvature.} }}

\subsubsection{The two definitions of $\inddiff$}

Let again $W^d$ be a $\Spin$-manifold and let $h\in\calR(\partial W)$ such that $\calF(W)_h\not=\emptyset$. Since $\calF$ implies positive scalar curvature we deduce that $h\in\calR_\psc(\partial W)$. Let us consider the trivial $W$-bundle over $\bbI\times\calF(W)_h\times\calF(W)_h$. A fiber-wise Riemannian metric $G$ is given by  $G_{(t,g_0,g_1)}\coloneqq\frac{1-t}2g_0 + \frac{1+t}{2}g_1$ in the fiber over $(t,g_0,g_1)$. The elongation has psc at infinity and invertible Dirac operators for $t=\pm1$. We therefore get an element
\[\inddiff\coloneqq\ind(\bbI\times\calF(W)_h\times\calF(W)_h\times W, G)\in \Omega\ko^{-d}(\calF(W)_h\times\calF(W)_h,\Delta),\]
where $\Delta$ denotes the diagonal. This is Hitchin's definition of the index-difference, cf.\ \cite{hitchin_spinors}. Fixing a base-point $g\in\calF(W)_h$ we obtain an element $\inddiff_g\in\Omega\ko^{-d}(\calF(W)_h, g)$ and a homotopy theoretic realization
\begin{align}\label{equ:inddiff}
	\inddiff^\calF_g\colon(\calF(W)_h,g)\to (\Omega^{\infty+d+1}\ko,*).
\end{align}
\begin{remark}\label{rem:compare-f-to-psc}
Note that the definition of $\ind$ only depends on the Dirac-operator associated to the metric and hence we get the following homotopy commutative triangle
\begin{center}
\begin{tikzpicture}
\node (0) at (0,1.2){$(\calF(W)_h,g)$};
\node (1) at (0,0){$(\calR_\psc(W)_h,g)$};
\node (2) at (6,0) {$(\Omega^{\infty+d+1}\ko,*)$};

\draw[right hook->] (0) to (1);
\draw[->] (1) to node[below]{$\inddiff^{\calR_\psc}_g$} (2);
\draw[->] (0) to node[above, sloped]{$\inddiff^\calF_g$} (2);
\end{tikzpicture}
\end{center}
\end{remark}
\noindent The second definition of the index-difference goes back to Gromov--Lawson \cite{gromovlawson_fundamentalgroup}. Let $W^d$ be a closed $\Spin$-manifold and consider the trivial $\bbR\times W$-bundle over $\calF(W)\times\calF(W)$. Choose a  smooth function $\psi\colon\bbR\to [0,1]$ that is constantly equal to $1$ on $[1,\infty)$ and equal to $0$ on $(-\infty,0]$. We get a fiber-wise metric $G\coloneqq dt^2 + (1-\psi(t))g_0 + \psi(t)g_1$ in the fiber over $(g_0,g_1)$. This has positive scalar curvature at infinity  and we hence get an element
\[\inddiff^{GL}\coloneqq\ind(\calF(W)\times\calF(W)\times\bbR\times W, G)\in\ko^{-d-1}(\calF(W)\times\calF(W),\Delta).\]
Again, after fixing a base point $g\in\calF(W)$ we get $\inddiff^{GL}_g\in\ko^{-d-1}(\calF(W),g)$ and a homotopy theoretic realization
\begin{align}\label{equ:inddiffgl}
	\inddiff^{GL,\calF}_g\colon(\calF(W)_h,g)\to (\Omega^{\infty+d+1}\ko,*).
\end{align}

\begin{thm}[Spectral flow index theorem {\cite[Theorem A]{ebert_inddiff}}]\label{thm:spectral-flow}
	For every closed $\Spin$-manifold $W$ the maps (\ref{equ:inddiff}) and (\ref{equ:inddiffgl}) are weakly homotopic.
\end{thm}

\begin{remark}
	Recall that two maps $f_0,f_1\colon X\to Y$ are called \emph{weakly homotopic} if for every map $\alpha\colon K\to X$ from a finite $CW$-complex $K$ we have that $f_0\circ \alpha$ is homotopic to $f_1\circ \alpha$. Weakly homotopic maps induce equal maps on homotopy groups.
\end{remark}

\subsubsection{The additivity  theorem for the index-difference}
One of the  main tools for computing the index-difference is the additivity theorem. In order to state it we need some notation. Let $X$ be a paracompact Hausdorff space and let $E\to X$, $E'\to X$ be two Riemannian $\Spin$-manifold bundles of fiber dimension $d$. Let $g,g'$ be metrics with cylindrical ends  such that $E$ and $E'$ have psc at infinity. Assume that  there exist functions $a_0<a_1\colon X\to \bbR$ such that $E$ and $E'$ are cylindrical over $X\times(a_0,a_1)$ and agree there $E_{(a_0,a_1)} = E'_{(a_0,a_1)}$. Assume that the Dirac operators are invertible over $E_{(a_0,a_1)}$. Let
\[E_0\coloneqq E_{(-\infty,a_1)}; \quad E_1\coloneqq E_{(a_0,\infty)}; \quad E_2\coloneqq E'_{(-\infty,a_1)};\quad E_2\coloneqq E'_{(a_0,\infty)}\]
and for $(i,j)\in\{(0,1),(2,3),(0,3),(1,2)\}$ let $E_{ij}\coloneqq E_i\cup E_j$.
\begin{thm}[{\cite[Theorem 3.12]{berw}}]
	\[\ind(E_{01}) + \ind(E_{23}) = \ind(E_{03}) + \ind(E_{12})\in\ko^{-d}(X)\]
	If $E_{01}$ and $E_{23}$ have invertible Dirac operators over a closed subspace $Y\subset X$, then this equation holds in $\ko^{-d}(X,Y)$.
\end{thm}


\noindent There is the following restatement in terms of the index-difference.

\begin{thm}[{\cite[Theorem 3.16]{berw}}]\label{thm:additivity}
	Let $M_0\overset V\leadsto M_1\overset W\leadsto M_2$ be $\Spin$-cobordisms, $h_i\in\calR(M_i)$ and $g\in\calF(V)_{h_0,h_1}$, $g'\in\calF(W)_{h_1,h_2}$. Then the following diagram commutes up to homotopy:
	\begin{center}
	\begin{tikzpicture}
		\node (0) at (0,1.2) {$\calF(V)_{h_0,h_1} \times \calF(W)_{h_1,h_2}$};
		\node (1) at (5,1.2) {$\calF(V\cup W)_{h_0,h_2}$};
		\node (2) at (0,0) {$\Omega^{\infty+d+1}\ko\times \Omega^{\infty+d+1}\ko$};
		\node (3) at (5,0) {$\Omega^{\infty+d+1}\ko$};

		\draw[->](0) to node[above]{$\_\cup\_$} (1);
		\draw[->](0) to node[left]{$\inddiff_g\times\inddiff_{g'}$} (2);
		\draw[->](1) to node[right]{$\inddiff_{g\cup g'}$} (3);
		\draw[->](2) to node[above]{$ + $} (3);
	\end{tikzpicture}
	\end{center}
\end{thm}

\subsubsection{The index-difference in an abstract setting}

Let $I\coloneqq[0,1]$ denote the interval and let $f\colon (X,x_0)\to (Y,y_0)$ be a pointed map. The mapping cylinder $\cyl(f)$ of $f$ is defined to be the space $(X\times[0,1]\amalg Y)/(x,1)\sim f(x)$ and let $i\colon Y\to \cyl(F)$ be the inclusion. We write
\[\ko^{-p}(f)\coloneqq\ko^{-p}(\cyl(f), X\times\{0\}\cup \{x_0\}\times I).\]
The homotopy fiber of $f$ at $y$ is defined as
\[\hofib_y(f)\coloneqq\{(x,c)\;|\;x\in X,\; c\colon[0,1]\to Y,\; c(0)=f(x),\; c(1) = y\}\]
and we have the canonical map $\epsilon_{y_0}\colon f^{-1}(y_0)\to\hofib_{y_0}(f),\;x\mapsto (x,\cst_{y_0})$ which is a pointed map by considering $*\coloneqq(x_0,\cst_{y_0})$ as the base-point of $\hofib_{y_0}(f)$. There is a natural map
\begin{align*}
	\eta_{y_0}\colon &(\bbI,\partial\bbI) \times(\hofib_{y_0}(f),*)\to(\cyl(f), X\times\{0\}\cup \{x_0\}\times[0,1])\\
		&(t,x,c)\mapsto\begin{cases}
				(x,1+t) &\quad \text{ if } t\le 0\\
				i(c(t)) &\quad\text{ if } t\ge0.
			\end{cases}
\end{align*}
Note that $\eta_{y_0}\circ(\id_{\bbI}\times\epsilon_{y_0})$ is homotopic to $\iota_{y_0}\colon (t,x)\mapsto(x, \frac12(t+1))$ as a map of pairs. The \emph{fiber transport map} is defined as
\begin{align*}
	\tau\colon&\Omega_{y_0}Y\to \hofib_{y_0}(f)\\
		&c\mapsto (x_0,c).
\end{align*}
For a class $\alpha\in\ko^{-p}(f)$ there is an associated \emph{base class} $\bas(\alpha)\coloneqq i^*\alpha\in\ko^{-p}(Y)$ and a \emph{transgression} $\trg(\alpha)\coloneqq \eta_{y_0}^*\alpha\in\Omega \ko^{-p}(\hofib_{y_0}(f))$. The loop map is defined by $l\colon\bbI\times\Omega_{y_0} Y\to Y, \; (t,c)\mapsto c(\frac12(t+1))$ and we write $\Omega\coloneqq l^*\colon\ko^{-p}(Y,y_0)\to\Omega \ko^{-p}(\Omega_{y_0}Y, \cst_{y_0})$.
\begin{lem}[{\cite[Lemma 3.19]{berw}}]\label{lem:abstractfiber}
	We have \[\tau^*\trg(\alpha) = \Omega\bas(\alpha)\in\Omega\ko^{-p}(\Omega_{y_0}Y,\cst_{y_0}).\]
\end{lem}
\noindent This lemma can be illustrated by the following homotopy-commutative diagram
\begin{center}
\begin{tikzpicture}
	\node (0) at (0,1.2) {$\Omega_{y_0} Y$};
	\node (1) at (5,1.2) {$\hofib_{y_0}(f)$};
	\node (2) at (2.5,0) {$\Omega^{\infty+p+1}\ko$};

	\draw[->] (0) to node[above]{$\tau$} (1);
	\draw[->] (0) to node[left,yshift=-5pt]{$\Omega\bas(\alpha)$} (2);
	\draw[->] (1) to node[right,yshift=-5pt]{$\trg(\alpha)$} (2);
\end{tikzpicture}
\end{center}
An instructive way to think about these class proposed in \cite{berw} is the following: $j^*\alpha$ for $j\colon X\times[0,1]\embeds \cyl(f)$ is a concordance in the sense of \pref{Definition}{def:concordance} and $j^*\alpha|_{X\times\{0\}}$ is acyclic. Hence the class $f^*\bas(\alpha)=j^*\alpha|_{X\times\{1\}}=0\in \ko^{-p}(X,x_0)$ and hence we get the following homotopy-commutative diagram where the columns are homotopy fiber sequences:
\begin{center}
\begin{tikzpicture}
	\node (0) at (0,2.4) {$\hofib_{y_0}(f)$};
	\node (1) at (3,2.4) {$\Omega^{\infty+p+1}\ko$};
	\node (2) at (0,1.2) {$X$};
	\node (3) at (3,1.2) {$*$};
	\node (4) at (0,0) {$Y$};
	\node (5) at (3,0) {$\Omega^{\infty+p}\ko.$};

	\draw[->] (0) to node[above]{$\trg(\alpha)$} (1);
	\draw[->] (2) to node[above]{} (3);
	\draw[->] (4) to node[above]{$\bas(\alpha)$} (5);

	\draw[->] (0) to node[above]{} (2);
	\draw[->] (2) to node[left]{$f$} (4);
	\draw[->] (1) to node[above]{} (3);
	\draw[->] (3) to node[above]{} (5);
\end{tikzpicture}
\end{center}

\subsubsection{Increasing the dimension}
As a consequence of the abstract setting described in the previous section we can now derive the following propagation result allowing us to increase the dimension. For this we further assume that $\calF$ is fibrant.
\begin{thm}[{\cite[Theorem 3.22]{berw}}]\label{thm:increase}
	Let $W$ be a $\Spin$-manifold of dimension $d$, $h\in\calF(\partial W)$ and $g\in\calF(W)_h$. Then the following diagram is weakly homotopy commutative
	\begin{center}
	\begin{tikzpicture}
		\node(0) at (0,1.2) {$\Omega_h\calF(\partial W)$};
		\node(1) at (5,1.2) {$\calF(W)_h$};
		\node(2) at (2.5,0) {$\Omega^{\infty + d + 1}\ko$};

		\draw[->] (0) to node[above]{$T$} (1);
		\draw[->] (0) to node[left, yshift=-5pt]{$-\Omega\inddiff_h$} (2);
		\draw[->] (1) to node[right, yshift=-5pt]{$\inddiff_g$} (2);
	\end{tikzpicture}
	\end{center}
	where $T$ denotes the fiber transport map after identifying $\calF(W)_h$ with $\hofib_h(\res)$ via $\epsilon_{h_0}$.
\end{thm}

\begin{proof}
Let $g_0\in\calF(W)_{h_0}$. We define $\sigma_{g_0}\colon\calF(\partial W)\to \calR(W)$ by
\[\sigma_{g_0}(h)\coloneqq a\cdot(h+dt^2) + (1-a)g_0\]
for some cutoff-function supported on the collar with $a|_{\partial W} \equiv1$. Note that $\sigma_{g_0}(h)\in\calR(W)_h$ and $\sigma_{g_0}(h_0) = g_0$. We define a fiber-wise Riemannian metric $m$ on the trivial $W$-bundle $W\times\cyl(\res)\to \cyl(\res)$ as follows: over $h\in\calF(\partial W)\subset\cyl(\res)$ let $m\coloneqq\sigma_{g_0}(h)$ and over $(g,t)\in\calF(W)\times[0,1]$ let $m\coloneqq t\cdot\sigma(\res(g)) + (1-t)\cdot g$. We note the following properties of $m$:
\begin{enumerate}
	\item $m|_{\cyl(\res)\times\partial W}$ has invertible Dirac operator.
	\item $m_{(g,0)} = g$ for all $g\in \calF(W)$.
	\item $m_{(g_0,t)}=g_0$ for all $t\in[0,1]$
\end{enumerate}
Hence, we can define $\beta\coloneqq (W\times\cyl(\res),m)\in\ko^{-d}(\cyl(\res),\calF(W)\times\{0\}\cup\{g_0\}\times[0,1])$. Since the choice of the cutoff function $a$ is convex, $\beta$ only depends on $g_0$. By \cite[Proposition 3.23]{berw}, we have that \footnote{The first statement is proven straightforward whereas the second statement is more involved and uses the \pref{spectral flow index theorem}{thm:spectral-flow} (cf.\ \cite[Theorem 3.10]{berw} and \cite{ebert_inddiff}).}
\begin{enumerate}
	\item $\epsilon_{h_0}^*(\trg(\beta))\colon\calF(W)_{h_0}\to\Omega^{\infty + d + 1}\ko$ is homotopic to $-\inddiff_{g_0}$.
	\item $\bas(\beta) \colon \calF(\partial W)\to\Omega^{\infty +d}\ko$ is weakly homotopic to $\inddiff_{h_0}$.
\end{enumerate}
 By \pref{Lemma}{lem:abstractfiber} we conclude that
\[\Omega\inddiff_{h_0}\sim \Omega\bas(\beta)\sim\tau^*\trg(\beta) \sim \underbrace{(\epsilon_{h_0}^{-1}\circ \tau)}_{=T}{}^*\inddiff_{g_0},\]
which finishes the proof of \pref{Theorem}{thm:increase}.
\end{proof}

\subsubsection{Relating $\inddiff$ to $\ind$ }\label{sec:inddiff-to-ind}

Let $W$ be a $d$-dimensional $\Spin$-manifold with boundary $M$, such that $(W,M)$ is $1$-connected. Let $\pi\colon E\to X$ be a smooth fiber bundle with fiber $W$ over a paracompact base $X$ and associated structure group $\diff_\partial(W)$. This as a $\Spin$-structure on the vertical tangent bundle which is constant along the boundary sub-bundle. Let $h_0\in\calR(M)$ be a fixed boundary condition such that $\calF(W)_{h_0}\not=\emptyset$. We get an associated fiber bundle
\[p\colon Q\times_{\diff_\partial(W)} \calF(W)_{h_0}\too X\]
Let $x_0\in X$ be a base point and let us identify $\pi^{-1}(x_0) = W$. Then $p^{-1}(x_0)$ can be identified with $\calF(W)_{h_0}$ and we choose a base point $g_0\in p^{-1}(x_0)$.

We will now construct an element $\beta\in\ko^{-d}(p)$, depending only on the bundle $\pi$ and the metric $g_0$. Let $k$ be a fiber-wise Riemannian metric on $\pi$ such that
\begin{enumerate}
	\item the restriction of $k$ to $\pi^{-1}(x_0) = W$ is equal to $g_0$,
	\item near  the boundary sub-bundle $\partial E$, the restriction of $k$ is a cylinder on $h_0$.
\end{enumerate}
Such a metric can be constructed using a partition of unity and $k$ is not assumed to be in $\calF(\pi^{-1}(x))_{h_0}$ for all $x$. Let $\tilde E\coloneqq \pr^*E$ for the natural map $\pr\colon \cyl(p)\to X$. $\tilde E$ then inherits a fiber-wise metric $\tilde k$ as follows: over $x\in X\subset\cyl(p)$ we take $\tilde k_x \coloneqq k_x$ and over a point $(x,g,t)\in Q\times_{\diff_\partial(W)} \calF(W)_{h_0}\times[0,1]$ we let $\tilde k_x\coloneqq(1-t)g + tk_x$. Then $\tilde k$ also satisfies boundary condition $h_0$ and it has positive scalar curvature for $t=0$ and $(x,g) = (x_0,g_0)$. Since $\tilde E$ also has a $\Spin$-structure on the vertical tangent bundle there is a family Dirac operator and hence a well-defined class $\beta\in\ko^{-d}(p)$. This has the following properties.
\begin{prop}[{\cite[Proposition 3.33]{berw}}]\label{prop:inddiff-to-ind}\leavevmode
\begin{enumerate}
	\item $\bas(\beta) = \ind(E,k)\in \ko^{-d}(X)$.
	\item $\trg(\beta) = \inddiff_{g_0}\in\Omega\ko^{-d}(\calF(W)_{h_0})$.
	\item $\beta$ is natural with respect to fiber bundles.
	\item Let $V\colon M\to M'$ is a $\Spin$-cobordism and $m\in\calF(V)_{h_0,h_1}$. Let
		\[\pi'\colon E\cup_\partial (X\times V) \to X\]
		be the bundle obtained by gluing in $V$ in each fiber. Then there is a commutative diagram
	\begin{center}
	\begin{tikzpicture}
		\node(0) at (0,1.2) {$Q\times_{\diff_\partial(W)} \calF(W)_{h_0}$};
		\node(1) at (5,1.2) {$Q\times_{\diff_\partial(W)} \calF(W\cup V)_{h_1}$};
		\node(2) at (2.5,0) {$X$};

		\draw[->] (0) to node[above]{$\mu_m$} (1);
		\draw[->] (0) to node[left, yshift=-5pt]{$p$} (2);
		\draw[->] (1) to node[right, yshift=-5pt]{$p'$} (2);
	\end{tikzpicture}
	\end{center}
	and the image of $\beta'\in \ko^{-d}(p')\to\ko^{-d}(p)$ agrees with $\beta$.
\end{enumerate}
\end{prop}

\section{Proof of main results}\label{sec:main}

\noindent For this entire section let $\calF$ be a parametrized codimension $c\ge3$ surgery stable, cellular Riemannian functor that implies positive scalar curvature.

\subsection{The action of the diffeomorphism group}\label{sec:diffeo-action}

\noindent Now let $M^{d}$ be a manifold with boundary $\partial M$ such that $\calF(M)\not=\emptyset$ and let $h\in\calR(\partial M)$ such that $h+dt^2\in\calF(\partial M\times[0,1])$. The space $\calF(M)_h$ admits an action of $\diff_\partial(M)$, the group of diffeomorphisms which are the identity on a neighborhood of $\partial M$, via pullback. We get an action map $\eta\colon\diff_\partial(M)\to \haut(\calF(M)_h)$ which induces
\begin{equation}
	\Gamma_\partial(M)\coloneqq \pi_0(\diff_\partial(M)) \too \pi_0(\haut(\calF(M)_h))\label{equ:action-map-boundary}.
\end{equation}
$\Gamma_\partial(M)$ is called the \emph{mapping class group} of $M$.

\begin{thm}\label{thm:action-map-boundary}
	Let $d\ge 2c-1$ and let $M^{d}$ be a $(c-2)$-connected, $B\ort(d)\langle c-1\rangle$-manifold with boundary $\partial M = S^{d-1}$. Also, assume that $M$ is $B\ort(d)\langle c-1\rangle$-cobordant to $D^{d-1}$ relative to the boundary. Then the image of the map (\ref{equ:action-map-boundary}) for $h=g_\circ^{d-2}$ is an abelian group.
\end{thm}

\noindent For the proof we will use the following Lemma of Eckmann-Hilton style.

\begin{lem}[{\cite[Lemma 4.2]{berw}}]\label{lem:eckmann-hilton}
	Let $\calC$ be a nonunital topological category with objects the integers and let $G$ be a topological group which acts on $\calC$, i.e. $G$ acts on all morphism spaces and the composition in $\calC$ is $G$-equivariant. We will denote the composition of $x$ and $y$ by $x\cdot y$. Suppose that
	\begin{enumerate}
		\item $\calC(m,n)=\emptyset$ for $n\le m$.
		\item For each $m\ne0$ there exists a $u_m\in \calC(m,m+1)$ such that the composition maps
			\begin{align*}
				u_m\cdot\_&\colon\calC(m+1,n)\to \calC(m,n) &\text{for } n>m+1\\
				\_\cdot u_m&\colon\calC(n,m)\to \calC(n, m+1) &\text{for } n<m
			\end{align*}
			are homotopy equivalences.
		\item There exists an $x_0\in \calC(0,1)$ such that the composition maps
			\begin{align*}
				x_0\cdot\_&\colon\calC(1,n)\to \calC(0,n) &\text{for } n>1\\
				\_\cdot x_0&\colon\calC(n,0)\to \calC(n, 1) &\text{for } n<0
			\end{align*}
			are homotopy equivalences.
		\item The $G$-action is trivial unless $m\le0$ and $1\le n$.
	\end{enumerate}
	Then for $f,g\in G$ the maps $f,g\colon \calC(0,1)\to\calC(0,1)$ commute up to homotopy.
\end{lem}

\begin{proof}[Proof of \pref{Theorem}{thm:action-map-boundary}]
	This is analogous to \cite[Proof of Theorem 4.1]{berw}. Consider a closed disk $\iota\colon D\subset S^{d-1}\times (0,1)$ in the interior. Since $\calF$ is cellular, surgery stable, there exists a metric $h\in\calF(S^{d-1}\times[0,1], \iota)_{g_\circ,g_\circ}$ which is isotopic to the product metric $g_\circ +dt^2$ relative to the boundary. Let $T\coloneqq (S^{d-1}\times [0,1])\setminus \mathrm{int}(\iota(D))$ and we denote by $P = S^{d-1}$ the boundary component created by cutting out $D$. Furthermore, let $\overline h\in\calF(T)$ be the metric obtained by cutting out the metric $\iota_*g^{0,1}$ on $\iota(D)$, where $g^{0,1}$ is the metric from \pref{Proposition}{prop:gluingdisk}. $\overline h$ restricts to the round metric on all three boundary components. We get the sequence of maps
	\[\calF(M)_{g_\circ}\overset{\mu(\_,{\overline h})}\too\calF(M\cup_{S^{d-1}\times\{0\}} T)_{g_\circ,g_\circ}\overset{\mu(\_,{\iota_*g^{0,1}})}\too\calF(M\cup_{S^{d-1}\times\{0\}} T\cup_P \iota(D))_{g_\circ}.\]
	The composition is given by gluing in $h$ which is homotopic to gluing in $g_\circ+dt^2$ and therefore is a homotopy equivalence. The right-most map is a homotopy equivalence by \pref{Proposition}{prop:gluingdisk} and so $\mu_{\overline h}$ also is a homotopy equivalence, too. Let $V:=M\cup_{S^{d-1}\times\{0\}} T$ and let us consider this as a cobordism $S^{d-1} = P \leadsto S^{d-1}\times\{1\} = S^{d-1}$.

	We now apply \pref{Lemma}{lem:eckmann-hilton} to the following scenario: Let $G:=\diff_\partial(M)$ and let $\calC(0,1) = \calF(V)_{g_\circ,g_\circ}$. Furthermore, let
{\small	\[\calC(m,n) = \begin{cases}
		\calF(S^{d-1}\times[m,0]\cup V\cup S^{d-2}\times [0,n])_{g_\circ,g_\circ}	&\text{for } m\le0, n\ge1\\
		\calF(S^{d-1}\times[m,n])_{g_\circ,g_\circ}	&\text{for } m<n\le0\text{ or } n>m\ge1\\
		\emptyset
	\end{cases}.\]}
	\noindent Let $G$ act on $\calC(m,n)$ by extending a diffeomorphism $f\in\diff_\partial(M)$ by the identity and then acting via pullback, \ie G acts on $M$ via pullback and trivially everywhere else. With this action the composition given by gluing metrics is obviously $G$-equivariant. For $m\ne0$ let $u_m\coloneqq g_\circ^{d-2} + dt^2\in \calC(m,m+1)$ and by \pref{Lemma}{lem:stable-metrics} there exists an $x_0\in\calC(0,1)$ such that the hypothesis of \pref{Lemma}{lem:eckmann-hilton} is satisfied. Thus, the action of $\diff_\partial(M)$ on $\calF(V)_{g_\circ,g_\circ}$ factors through an abelian group. The Theorem follows because the gluing map $\mu(\_,{\overline h})\colon \calF(M)_{g_\circ}\to\calF(V)_{g_\circ,g_\circ}$ is a $\diff_\partial(M)$-equivariant homotopy equivalence.
\end{proof}

\subsection{Construction of maps}\label{sec:construction}

\noindent We will now construct maps into $\calF(M)$ for certain even-dimensional manifolds. This is similar to \cite[Section 4.2]{berw}. Let $n\ge c$ and let $W\colon\emptyset\leadsto S^{2n-1}$ be a $(c-2)$-connected $\boc$ cobordism which is $\boc$-cobordant to $D^{2n}$ relative to the boundary. Let
\[K\coloneqq([0,1]\times S^{2n-1})\# S^n\times S^n\colon S^{2n-1}\leadsto S^{2n-1}.\]
For $i=0,1,\dots$ let $K\mid_{i}\coloneqq S^{2n-1}$ and $K\mid_{[i,i+1]}\colon K\mid_i\leadsto K\mid_{i+1}$ be a copy of $K$. If we consider $W\colon\emptyset\leadsto K\mid_{0}$ we can define
\[W_k\coloneqq W\cup \bigcup_{i=0}^{k-1} K\mid_{[i,i+1]}\colon \emptyset\leadsto K\mid_i.\]
We abbreviate $D_k\coloneqq \diff_\partial(W_k)$, $B_k\coloneqq BD_k$ and $\pi_k\colon E_k\coloneqq ED_k\times_{D_k} W_k\to B_k$. There is a homomorphism $D_k\to D_{k+1}$ given by extending by the identity and we get induced maps $\iota_k\colon B_k\to B_{k+1}$ on classifying spaces. Furthermore we write $\calF_k\coloneqq\calF(W_k)_{g_{\circ}}$, $T_k\coloneqq ED_k\times_{D_k}\calF_k$ and we denote by $p_k\colon T_k\to B_k$ the projection maps and by $\mu_k\coloneqq\mu(\_,h_k)\colon \calF_k\to\calF_{k+1}$ the maps gluing in the stable metrics $h_k\in\calF(K|_{[k,k+1]})_{g_\circ,g_\circ}$ which exist by \pref{Lemma}{lem:stable-metrics}. The map $\mu_k$ is $D_k$-equivariant and so there is an induced map between the Borel constructions
\begin{center}
\begin{tikzpicture}
	\node (0) at (0,2.4) {$\calF_k$};
	\node (1) at (3,2.4) {$\calF_{k+1}$};
	\node (2) at (0,1.2) {$T_k$};
	\node (3) at (3,1.2) {$T_{k+1}$};
	\node (4) at (0,0) {$B_k$};
	\node (5) at (3,0) {$B_{k+1}$};

	\draw[->] (0) to node[above]{$\mu_k$} (1);
	\draw[->] (2) to node[above]{} (3);
	\draw[->] (4) to node[above]{$\lambda_k$} (5);

	\draw[->] (0) to node[above]{} (2);
	\draw[->] (2) to node[above]{} (4);
	\draw[->] (1) to node[above]{} (3);
	\draw[->] (3) to node[above]{} (5);
\end{tikzpicture}
\end{center}
We introduce the following notation:
\begin{align*}
	B_\infty&\coloneqq\mathrm{hocolim}_k B_k,	&& \iota_{k,\infty}\colon B_k\to B_\infty,\\
	\calF_\infty&\coloneqq\mathrm{hocolim}_k \calF_k &&\mu_{k,\infty}\colon \calF_k\to \calF_\infty\\
	T_\infty&\coloneqq\mathrm{hocolim}_k T_k, 	&& p_\infty \coloneqq\underset{k}{\mathrm{hocolim}}\ p_k\colon T_\infty\to B_\infty.
\end{align*}
The construction (\ref{sec:inddiff-to-ind}) gives classes $\beta_k\in\ko^{-2n}(p_k)$ that assemble to a class $\beta_\infty\in\ko^{-2n}(p_\infty)$ (cf.\ \cite[Proposition 4.9]{berw}).

\begin{lem}\label{lem:assumption}
	There exists a cobordism $W\colon \emptyset\leadsto S^{2n-1}$ such that there exists an acyclic map $\Psi\colon B_\infty\to \mtt_{c-1}(2n)$ and the maps
	\[\ind(E_k,h_\circ), \Omega^\infty\lambda_{-2n}\circ\Psi\circ\iota_{k,\infty}\colon B_k\too \Omega^{\infty+2n}\ko\]
	are weakly homotopic.
\end{lem}

\begin{rem}
	Recall that a map $f\colon X\to Y$ is called acyclic, if for each $y\in Y$ the homotopy fiber of $\hofib_y(f)$ has the singular homology of a point.
\end{rem}

\begin{proof}[Proof of \pref{Lemma}{lem:assumption}]
	By the analogue of \cite[Proposition 4.20]{berw} there exists a $B\ort(2n)\scpr{c-1}$-cobordism $W\colon \emptyset\leadsto S^{2n-1}$ which itself is $B\ort(2n)\scpr{c-1}$-cobordant to $D^{2n}$ and such that the structure map $W\to B\ort(2n)\scpr{c-1}$ is $n$-connected. The bundles $\pi_k$ from above yield a Pontryagin--Thom map
	\[\alpha_k\colon B_k\to \Omega^\infty\mtt_{c-1}(d)\]
	and since the $\Spin$-structures on all bundles $\pi_k$ are compatible, we obtain a map
	\[\alpha_\infty\colon B_\infty\to \Omega^\infty\mtt_{c-1}(d).\]
	The map $\alpha_\infty$ is acyclic for our choice of $W$ which follows from \cite[Theorem 1.5]{grw_stability} in the same way as demonstrated in the proof of \cite[Theorem 4.19]{berw}. 
	
	It remains to show that the maps $\iota_{k,\infty}^*\alpha_\infty^*\Omega^\infty\lambda_{-2n}$ and $\ind(E_k, h_\circ)$ are weakly homotopic. For this note that a $W$-bundle $E\to X$ with a $\theta_{c-1}$-structure on the vertical tangent bundle also admits a $\Spin$-structure and we get the following diagram
	\begin{center}
	\begin{tikzpicture}
		\node (0) at (0,0) {$X$};
		\node (1) at (4,1.5) {$\Omega^{\infty}\mtt_{c-1}(d)$};
		\node (2) at (4,0) {$\Omega^{\infty}\mtt_{2}(d)$};
		\node (3) at (8,0) {$\Omega^{\infty+d}\ko$}; 
		
		\draw[->] (0) to node[above, sloped]{$\alpha^\theta_E$} (1);
		\draw[->] (0) to node[above, xshift=10pt]{$\alpha^\Spin_E$} (2);
		\draw[->] (1) to node[above, sloped]{$\Omega^\infty\lambda_{-2n}$} (3);
		\draw[->] (2) to node[above,  xshift=-10pt]{$\Omega^\infty\lambda_{-2n}$} (3); 
		\draw[right hook->] (1) to (2);
		\draw[->, bend right=20] (0) to node[above, sloped]{$\ind(E,h_\circ)$} (3);
	\end{tikzpicture}
	\end{center}
	The maps in the bottom triangle are weakly homotopic by \cite[Proposition 4.16]{berw} which is a version of the Atiyah--Singer index theorem. The left-hand triangle commutes since the $\Spin$-structure on $T_{(v)}E$ is precisely the one induced by the $\theta_{c-1}$-structure and the right-hand triangle commutes by definition. Therefore the entire diagram commutes and hence we obtain that $\Omega^\infty\lambda_{-2n}\circ\alpha^\theta_E$ and $\ind(E,h_\circ)$ are weakly homotopy which specifies to our claim for $X=B_\infty$.	
\end{proof}

\noindent\emph{Let us from now on abbreviate $X\coloneqq\mtt_{c-1}(2n)$.} We will now proceed to construct a map $\rho_\Psi\colon\Omega X\to\calF(W)_{g_\circ}$ out of $\Psi$. We have the following diagrams
\begin{center}
\begin{tikzpicture}
	\node (0) at (0,1.2) {$B_\infty$};
	\node (1) at (3,1.2) {$X$};
	\node (2) at (3,0) {$B\haut(\calF_\infty)$};

	\draw[->, dotted] (1) to node[right] {$\gamma$} (2);
	\draw[->] (0) to node[above]{$\Psi$} (1);
	\draw[->] (0) to node[below, xshift=-4pt]{$B\eta$} (2);
\end{tikzpicture}
\end{center}
where $\eta$ is the action map. The group $\ker(\pi_1(\Psi))$ is perfect since $\Psi$ is acyclic and $\pi_1(B\eta)$ has abelian image by \pref{Theorem}{thm:action-map-boundary}. Hence $\pi_1(B\eta)(\ker(\pi_1(\Psi)))$ is trivial and so $\ker(\pi_1(\Psi))\subset \ker(B\eta)$. By \cite[Proposition 3.1]{hausmannhusemoller} the acyclicity of $\Psi$ implies that the map $\gamma$ exists and is unique up to homotopy.

 Let $p_\infty^+\colon T_\infty^+\to X$ denote the fibration obtained by pulling back the universal fibration $E\to B\haut(\calF_\infty)$ with fiber $\calF_\infty$ along $\gamma$. We get an induced commutative diagram of fibrations
\begin{center}
\begin{tikzpicture}
	\node (0) at (0,1.2) {$T_\infty$};
	\node (1) at (3,1.2) {$T_\infty^+$};
	\node (2) at (6,1.2) {$E$};
	\node (3) at (0,0) {$B_\infty$};
	\node (4) at (3,0) {$X$};
	\node (5) at (6,0) {$B\haut(\calF_\infty)$};

	\draw[->] (0) to node[auto]{$\hat\Psi$} (1);
	\draw[->] (1) to (2);
	\draw[->] (0) to node[auto]{$p_\infty$}(3);
	\draw[->] (1) to node[auto]{$p_\infty^+$} (4);
	\draw[->] (2) to (5);
	\draw[->] (3) to node[auto]{$\Psi$}(4);
	\draw[->] (4) to node[auto]{$\gamma$}(5);
\end{tikzpicture}
\end{center}
where all fibers are homotopy equivalent to $\calF_\infty$. The homotopy fibers of $\Psi$ and $\hat\Psi$ are homotopy equivalent and hence $\hat\Psi$ is acyclic. Therefore, the class $\beta_\infty$ extends to a unique class $\beta^+_\infty\in\ko^{-2n}(p_\infty^+)$. Now let $\rho_\infty\colon \Omega X\to\calF_\infty$ denote the fiber transport of the fibration $p_\infty^+$. Since $\mu_{0,\infty}\colon\calF_0\to\calF_\infty$ is a homotopy equivalence we obtain a map
\[\rho_\Psi\colon \Omega X\to \calF_0 = \calF(W)_{g_\circ}\]
that satisfies $\Omega\bas(\beta_\infty^+) = \rho_\Psi^*\mu_{0,\infty}^*\trg(\beta_\infty^+)$. We have that the following maps are weakly homotopic
\begin{equation}
\bas(\beta_\infty^+)\circ\Psi\circ \iota_{k,\infty} \overset{\ref{prop:inddiff-to-ind}}\sim \ind(E_k,g_\circ) \overset{\ref{lem:assumption}}\sim \lambda_{-2n}\circ\Psi\circ\iota_{k,\infty} \in \ko^{-2n}(B_k).\label{equ:1}
\end{equation}
Since $\Psi$ is acyclic, this implies that $\bas(\beta_\infty^+)\sim\lambda_{-2n}$ and hence $\Omega\bas(\beta_\infty^+)\sim\Omega\lambda_{-2n}\colon \Omega X\to\Omega^{\infty+2n+1}\ko$. Therefore we get the following equation of weakly homotopic maps:
\begin{align*}
	\inddiff_g\circ\rho_\Psi &\overset{\ref{prop:inddiff-to-ind}}\sim \trg\beta_0^+\circ\rho_\Psi \sim \trg\beta_\infty^+\circ \mu_\infty\circ \rho_\Psi\overset{\ref{lem:abstractfiber}}\sim \Omega\bas(\beta^+_\infty)\overset{(\ref{equ:1})}\sim\Omega\lambda_{-2n}
\end{align*}
and together with \pref{Remark}{rem:compare-f-to-psc} we arrive at the following result.
\begin{thm}\label{thm:mainforw}
	The map $\Omega^{\infty+1}\lambda_{-2n}$ and the composition
	\[\Omega^{\infty+1}\mtt_{c-1}(2n)\overset{\rho_\Psi}\too \calF(W)_{h_\circ}\embeds\calR_\psc(W)_{h_\circ}\overset{\inddiff_g}\too \Omega^{\infty+2n+1}\ko \]
	induce the same map on homotopy groups.
\end{thm}

\subsection{Propagating the detection result}\label{sec:propagation}
One consequence of the index additivity theorem is the following propagation result. Let $\calF$ be a cellular, parametrized codimension $c\ge3$ surgery stable Riemannian functor and implies positive scalar curvature. Let $d\ge 2c$ and let $W\colon\emptyset\leadsto S^{d-1}$ be a $(c-2)$-connected $B\ort\langle c-1\rangle$-manifold and $X\colon W\leadsto D^{d}$ a $B\ort\langle c-1\rangle$-cobordism relative to the boundary. By removing an embedded disk $D\subset W$ we obtain a bordism $W_0\colon S^{d-1}\leadsto S^{d-1}$ which is $B\ort\langle c-1\rangle$-bordant to the cylinder relative to the boundary via $X_0$. Therefore, there exists a stable metric $\tilde g$ on $W_0$ by \pref{Lemma}{lem:stable-metrics}. Let $g\coloneqq g^{0,1}\cup \tilde g\in \calF(W)_{h_\circ}$.

By performing surgery on $X_0$ we may assume that $X_0$ is $(c-2)$-connected and after choosing an appropriate handle decomposition of $X_0$ we get a surgery map $\calS_{\calF,X_0,H}$ (cf.\ \pref{Corollary}{cor:surgery-equivalence} and the discussion below it). Thus, we obtain a (homotopy class of a) metric $\tilde g\coloneqq\calS_{\calF,X_0, H}(g_\circ + dt^2)$ and $g\coloneqq g^{0,1} \cup \tilde g\in\calR^+(W)_{h_{\circ}}$ for $g^{0,1}$ the metric from \pref{Proposition}{prop:gluingdisk}.

\begin{prop}\label{prop:propagation}
Let $W'\colon \emptyset\leadsto M'$ be an arbitrary compact $\Spin$-cobordism with $h'\in\calR(M')$ and $g'\in\calF(W')_{h'}$. We get the following result. If there exists a $CW$ complex $X$ with a map $\hat a\colon X\to \Omega^{\infty+d+1}\ko$ and a factorization
	\[X\overset{\rho}\too \calF(W)_{h_\circ} \overset{\inddiff_g}\too \Omega^{\infty+d+1}\ko\]
	of $\hat a$ up to homotopy, then there exists a factorization
	\[X\overset{\rho'}\too \calF(W')_{h'} \overset{\inddiff_{g'}}\too \Omega^{\infty+d+1}\ko\]
	of $\hat a$ up to homotopy, too.
\end{prop}

\begin{proof}
	The map $\mu_{\tilde g}\colon\calF(D^{d})_{h_\circ}\to\calF(W)_{h_\circ}$ gluing in the metric $\tilde g$ is a weak homotopy equivalence by \pref{Lemma}{lem:stable-metrics}. Hence there exists a lift $\tilde\rho\colon X\to \calF(D^{d})_{h_\circ}$ of $\rho$ along $\mu_{\tilde g}$ and by the \pref{Additivity theorem}{thm:additivity} the composition
	\[X\overset{\tilde\rho}\too\calF(D^{d})_{h_\circ} \overset{\inddiff_{g^{0,1}}}\too\Omega^{\infty+d+1}\ko\]
	is again a factorization of $\hat a$ up to homotopy. By \pref{Proposition}{prop:gluingdisk} there exists a metric $\tilde g'$ on $W'\setminus D$ such that $g^{0,1}\cup \tilde g'$ is homotopic to $g'$. Again, by \pref{Theorem}{thm:additivity} the composition given by
	\[X\overset{\tilde\rho}\too\calF(D^{d})_{h_\circ}\overset{\mu_{\tilde g'}}\too\calF(W')_{h'}\overset{\inddiff_{g'}} \too\Omega^{\infty+d+1}\ko.\]
	is homotopic to $\hat a$.
\end{proof}

\begin{rem}
	Note that we do \emph{not} require the manifold $W'$ to admit a $B\ort(d)\scpr{c-1}$ or to be $(c-2)$-connected. The proof reveals that the map $\rho'$ from \pref{Proposition}{prop:propagation} actually factors through $\calF(D^{d})_{h_\circ}$. The requirement of $W'$ being $\Spin$ stems from requiring the existence of the map $\inddiff_{g'}$.
\end{rem}

\subsection{Proof of {Theorem \ref*{thm:maingeneral}}}\label{sec:assemble}

The proof of the general statement of our main results now consists of assembling the parts. First we note that the result for the manifold $W$ from \pref{Theorem}{thm:mainforw} follows directly from  \pref{Theorem}{thm:mainforw} and \pref{Theorem}{thm:rational-mtw}. Since  $W$ is $B\ort(d)\scpr{c-1}$-cobordant to $D^{2n}$ relative to the boundary, \pref{Proposition}{prop:propagation} implies that it is true for every manifold $W'$ of dimension $2n$, in particular it is true for $S^{2n}$. If $\calF$ is fibrant, then by \pref{Theorem}{thm:increase} it holds for $D^{2n+1}$ and again by \pref{Proposition}{prop:propagation} it holds for all manifolds of dimension $2n+1$.\qed

\section{Epilog}\label{sec:epilog}
\noindent\cite{berw} is not the only paper about the space of positive scalar curvature which geometrically mostly depends on the parametrized surgery theorem. We therefore believe, that many other recent results can be proven for positive $p$-curvature und $(d-k)$-positive Ricci curvature. The following gives a (probably very incomplete) list of results that could possibly be generalized, with adapted dimension and connectivity assumptions, of course:
\begin{enumerate}
	\item[\cite{botvinnikhankeschickwalsh}] Here it is shown that there exist elements of infinite order in the some homotopy groups of the observer moduli space $\calM^\psc_{x_0}(M)$. For $k\ge2$ and $k$-positive Ricci curvature this follows from \cite{botvinnikwalshwraith} and the fact, that $k$-positive Ricci curvature is codimension $0$ surgery stable. For positive $p$-curvature however this is not known.
	\item[\cite{walsh_hspaces}] Here it is shown that the component of the round metric in $\calR_\psc(S^{d})$ is a $d$-fold loop space for $d\ge3$. This has recently been upgraded in \cite{walshwraith} to be true for $k$-positive Ricci curvature, $k\ge2$ and also $d\ge3$. Again, an analogue for positive $p$-curvature does not exist, but we suspect that it's true for $d\ge3$ and positive $p$-curvature for $p\le d-3$.
	\item[\cite{erw_psc2}] Here, the existence of (left-)stable metrics (cf. \pref{Definition}{def:stable-metrics})  is shown on cobordisms where the inclusion of the outgoing boundary is $2$-connected. Also this contains an extension of the results from \cite{berw} taking the fundamental group into account.
	\item[\cite{erw_psc3}] Here it is shown that the component of the round metric in $\calR_\psc(S^{d})$ is an infinite loop space for $d\ge 6$. However, the given proof does not work in low dimensions and we expect the dimension and connectivity assumptions to be worse worse compared to the ones from \cite{walsh_hspaces}.
	\item[\cite{actionofmcg}] For manifolds $M,N$ of dimension $d\ge6$ it is shown that the surgery map (cf. \pref{Corollary}{cor:surgery-equivalence} and the discussion below for the definition)
		\[\calS_{\calR_\psc,X,H}\colon\calR_\psc(M)\to \calR_\psc(N)\]
		 is independent of the handle decomposition $H$ and only depends on the $\theta$-cobordism class of the cobordism $X$ relative to $M\amalg N$, for $\theta$ the tangential $2$-type of the outgoing boundary $N$. This is then utilized to show that for a simply connected $\Spin$-manifold of dimension at least $6$ the pullback action $\pi_0(\diff^\Spin(M))\actson\pi_0(\calR_\psc(M))$ factors through the $\Spin$-cobordism group.
	\item[\cite{hspaces}] Here it is shown using the results from \cite{actionofmcg} that $\calR_\psc(M)$ is an $H$-space for every manifold $M$ of dimension at least $6$ which is nullbordant in its own tangential $2$-type. It is also shown that in dimensions at least $6$ the underlying $H$-spaces structures from \cite{erw_psc3} and \cite{walsh_hspaces} are equivalent.
\end{enumerate}

\appendix

\section{Curvature Computations}

\begin{lem}\label{lem:pstable}
	If $g$ has positive $p$-curvature then $g+\dt^2$ does as well.
\end{lem}

\begin{proof}
  Let $(M,g)$ be a closed, $d$-dimensional Riemannian manifold.
  Let $P \subset \rmT_xM$ be a $p$-dimensional subspace and choose an orthonormal base $E_1, \ldots, E_{d-p}$ of $V^{\perp}$.
  Since $g$ has positive $p$-curvature we get:
  \begin{align*}
    s_{p,g}(V) = \sum_{i,j = 1}^{d-p} \sec(E_i,E_j) > 0.
  \end{align*}
  Again, let $W \subset \rmT_{(x,t)}(M \times [0,1])$ be a p-dimensional subspace and suppose $W^{\perp}$ is not fully contained in $\rmT_xM$ (otherwise positivity of $s_{p,g+\dt^2}(W)$ follows directly from positivity of $s_{p,g}$).
  Then choose an orthonormal base $E_1, \ldots, E_{d+1-p-1}$ for $W^{\perp} \cap \rmT_xM$ and extend this to an orthonormal basis of $W^{\perp}$ by a vector $E_{d+1-p} = \sin (\phi) X + \cos (\phi) Y$, where $\phi\in[0,\frac\pi2)$, $X \in \rmT_xM$ and $Y \in \rmT_t[0,1]$ are vectors of unit length.
  This yields the following description.
  \begin{align}\label{eq:p-curv-decomp}
    s_{p,g+\dt^2}(W)
    & = \sum_{i,j = 1}^{d+1-p} \sec(E_i,E_j)
    = \sum_{i,j = 1}^{d-p} \sec(E_i, E_j) + 2 \sum_{i = 1}^{d-p} \sec(E_i,E_{d+1-p}) \nonumber\\
    & = \sum_{i,j = 1}^{d-p} \sec(E_i, E_j) + 2 \sum_{i = 1}^{d-p} \sec(E_i,\sin (\phi) X + \cos (\phi) Y)
  \end{align}
  Next observe that for $1 \leq i < d-p$
  \begin{align*}
    0 = g(E_{d+1-p},E_i) = \sin (\phi) \,  g(X,E_i) + \cos (\phi) \, \underbrace{g(Y,E_i)}_{= 0}.
  \end{align*}

  \noindent\textbf{Case 1:} $\sin (\phi) = 0$.
  Here, certainly $E_{d+1-p} = Y \in \rmT_t[0,1]$ and we can conclude
  \begin{align*}
    s_{p,g+\dt^2}(W)
    \overset{\eqref{eq:p-curv-decomp}}{=} \sum_{i,j = 1}^{d-p} \sec(E_i, E_j) + 2 \sum_{i = 1}^{d-p}\underbrace{\sec(E_i,Y)}_{=0} > 0,
  \end{align*}
  because the first term is positive (it equals $s_{p,g}(\operatorname{span}(E_1, \ldots, E_{d-p})^{\perp}) > 0$).\vspace{0pt}

  \noindent\textbf{Case 2:} $\sin (\phi) \neq 0$, i.e.\ $g(X,E_i) = 0$.
  In this case we consider the latter term in \eqref{eq:p-curv-decomp}, which can be decomposed as follows.
  \begin{align*}
    \sec(E_i, \sin (\phi) X + \cos (\phi) Y)
    & = \sin^2 (\phi) \, \sec(E_i,X) + \cos^2 (\phi) \, \underbrace{\sec(E_i,Y)}_{= 0}\\
    & \quad + \sin (\phi) \, \cos (\phi) \, R(E_i,X,Y,E_i),
  \end{align*}
  for $1 \leq i \leq d-p$. In this equation the last term vanishes, because
  \begin{align*}
    R(E_i,X,Y,E_i) = -R(E_i,X,E_i,Y) = -g(\underbrace{R(E_i,X)E_i}_{\in T_pM},Y) = 0.
  \end{align*}
  and it follows that
  \begin{align*}
    s_{p,g+\dt^2}(W)
    = \underbrace{\sum_{i,j = 1}^{d-p} \sec(E_i, E_j)}_{=: a} + 2 \sin^2 (\phi) \, \underbrace{\sum_{i = 1}^{d-p} \sec(E_i,X)}_{=: b}.
  \end{align*}
  Observe that $a = s_{p,g}(\operatorname{span}(E_1, \ldots, E_{d-p})^{\perp}) > 0$.
  Because positive $p$-curvature on $M$ implies positive $(p-1)$-curvature, we have that
  \begin{align*}
    0   < s_{p-1,g}(\operatorname{span}(E_1, &\ldots, E_{d-p},X)^{\perp})\\
    	&= \sum_{i,j = 1}^{d-p} \sec(E_i,E_j) + 2 \sum_{i = 1}^{d-p} \sec(E_i,X)
    = a + 2b.
  \end{align*}
  A case distinction for the sign of $b$ leads to the conclusion that 
  $s_{p,g+\dt^2}(W) = a + 2\sin^2(\phi) \, b > 0$, as $0 < \sin^2(\phi) < 1$ and $a>0$.
\end{proof}

\begin{lem}\label{lem:computation-gajer}
With the notation from the proof of \pref{Lemma}{lem:gajer} we have
  \begin{align*}
    R_{(N \times \bbR, g_{f(t)} + \mathrm{d} t^2)}|_{(x,t_0)} & = R_{(N \times \bbR, g_{f(t_0)} + \mathrm{d} t^2)}\\
                                                              & \qquad  + O(|f'|) E_1 + O(|f'|^2) E_2 + O(|f''|) E_3.
  \end{align*}
\end{lem}

\begin{proof}
 Let $\overline x_0 = (x_0, t_0) \in N \times [0,1]$ and let $M := N \times [0,1]$, as well as $g := g_{f(t)} + \mathrm{d} t^2$. Then $\rmT_{\overline x_0} M = \rmT_{x_0}N \oplus \rmT_{t_0}[0,1]$.
Choose a normal coordinate frame $(X_1, \ldots, X_n, T)$ with $T = \frac{\partial}{\partial t}$ around $\overline x_0$. In particular we have from the definition of $g$ that
\begin{align}\label{eq:T-is-normal}
  g(X_i, T) = 0
  \quad
  \text{ and }
  \quad
  g(T, T) = 1.
\end{align}

\noindent Now consider the hypersurface $Y := N \times \{t_0\} \hookrightarrow (M, g)$ with the metric induced from $g$.
By definition, $g|_Y = g_{f(t_0)}$.
From the Gauß equation, we get
\begin{align*}
  R^{(Y, g|_Y)} = R^{(M, g)} + \mathrm{I\!I}_Y \wedge \mathrm{I\!I}_Y,
\end{align*}
where $\two_Y$ denotes the second fundamental form and we let $\mathrm{I\!I}_Y \wedge \mathrm{I\!I}_Y(a,b,c,d) = \mathrm{I\!I}_Y(a,d) \mathrm{I\!I}_Y(b,c) - \mathrm{I\!I}_Y(a,c) \mathrm{I\!I}_Y(b,d)$. Using this we get
\begin{align}\label{eq:gauss}
  R^M(X_i, X_j, X_k, X_l) = \underbrace{R^Y(X_i, X_j, X_k, X_l)}_{ = R^{(N, g_{f(t_0)}) \times \mathbb E}(X_i, X_j, X_k, X_l)} -\  \mathrm{I\!I}_Y \wedge \mathrm{I\!I}_Y(X_i, X_j, X_k, X_l)
\end{align}
\noindent Note that by \eqref{eq:T-is-normal}, $T$ is a normal field to $Y$.
In particular, we have $\mathrm{I\!I}_Y(A, B) = g(\nabla_A T, B) = - g(T, \nabla_AB)$. We obtain via Koszul (using \eqref{eq:T-is-normal} and the fact that in a coordinate frame the Lie brackets vanish)
\begin{align*}
  \mathrm{I\!I}_Y(X_i, X_j) & = -g(T, \nabla_{X_i} X_j)\\
                     & = -\frac{1}{2}\big( X_i\ccancel{g(X_j, T)} + X_j\ccancel{g(X_i,T)} - Tg(X_i,X_j) \big.\\
                     & \quad \big. + g(\ccancel{[X_i,X_j]}, T) - g(\ccancel{[X_i,T]},X_j) - g(\ccancel{[X_j,T]}, X_i)\big)\\
                     & = \frac{1}{2} Tg(X_i, X_j)
\end{align*}
\noindent Thus we can write
\begin{align*}
  Tg(X_i, X_j) = T(g_{f(t_0)}(X_i, X_j)) = \frac{\partial g_{f(t_0)}(X_i, X_j)}{\partial t}
  = f'(t_0) \cdot \underbrace{\frac{\partial g_r(X_i, X_j)}{\partial r}}_{=: C_{ij}}.
\end{align*}
Note that the constants $C_{ij} = C_{ij}(t_0)$ depend continuously on $t_0$ and the family $\{g_r\}$ and are independent of $f$. Furthermore, they are bounded in $t_0$ by compactness. At $\overline x_0$, we have\begin{align*}
  \big.\mathrm{I\!I}_Y(X_i, X_j)\big|_{\overline x_0}  & = \big. \frac{1}{2} Tg(X_i,X_j)\big|_{\overline x_0}
                                                  = \frac{1}{2} f'(t_0)\left.\frac{\partial}{\partial r}\right|_{r = f(t_0)}g_r(X_i, X_j)
\end{align*}
\noindent This is enough to see that in \eqref{eq:gauss} the  $\mathrm{I\!I}_Y\wedge\mathrm{I\!I}_Y$ term only contributes with a distortion of order $O(|f'|^2)$, i.e. we can write
\begin{align}\label{eq:Rijkl}
  R^{M}(X_i, X_j, X_k, X_l)|_{\overline x_0} = R^Y(X_i, X_j, X_k, X_l)_{(x_0, 0)} + O(|f'|^2)E,
\end{align}
where $E$ is an expression in $C_{il}, C_{jk}, C_{ik}, C_{jl}$. Now let us check the remaining terms. First, metricity of the Levi-Civita  connection shows
\begin{align}\label{eq:ijkT}
 \nonumber g(\nabla_{X_i}\nabla_{X_j}X_k, T) &= X_ig(\nabla_{X_j}X_k, T) - g(\nabla_{X_j}X_k, \nabla_{X_i}T)\\
  \nonumber & = X_i(X_j\ccancel{g(X_k, T)} - g(X_k, \nabla_{X_j}T)) - g(\nabla_{X_j}X_k, \nabla_{X_i}T)\\
  & = -X_ig(X_k, \nabla_{X_j}T) - g(\nabla_{X_j}X_k, \nabla_{X_i}T)
\end{align}

\noindent Furthermore, note that $\nabla_T T = 0$, since
\begin{align*}
  g(X_i, \nabla_T T) & = \frac{1}{2}(\ccancel{X_ig(T, T)} + T\ccancel{g(T,X_i)} - T\ccancel{g(X_i,T)}) = 0,\\
  g(T, \nabla_T T) & = \frac{1}{2}(\ccancel{Tg(T, T)} + \ccancel{Tg(T,T)} - \ccancel{Tg(T,T)}) = 0.
\end{align*}
Because of \eqref{eq:T-is-normal}, and using what we know so far
\begin{align*}
  g(X_k, \nabla_{X_i}T) & = X_i\ccancel{g(X_k,T)} - g(\nabla_{X_i}X_k, T) = \frac{1}{2} Tg(X_i, X_k) = \frac{1}{2} f' C_{ik}\\
  g(\nabla_T X_i, T) & = g(X_i, \nabla_T T) = 0
\end{align*}
A computation similar to the one in \eqref{eq:ijkT}, shows (we use the convention $X_{n+1} := T$)
\begin{align*}
  g(\nabla_T\nabla_{X_i}X_k, T)
  & = -Tg(X_k, \nabla_{X_i}T) - g(\nabla_{X_j}X_k, \ccancel{\nabla_TT})\\
  &  = -T(Tg(X_i,X_k))\\
  & = - f'' C_{ik}\\
  g(\nabla_{X_i}\nabla_TX_k, T)
  & = -X_ig(X_k, \ccancel{\nabla_TT}) - g(\nabla_{T}X_k, \nabla_{X_i}T)\\
  &  = - g(\nabla_{T}X_k, \nabla_{X_i}T)\\
  & = -g(\sum_l \Gamma_{kl}^{n+1} X_l, \sum_m \Gamma^i_{n+1,m}X_m)
\end{align*}
Here we note that
\begin{align*}
  \Gamma^i_{n+1,m} & = \frac{1}{2}\sum_a g^{ia}(Tg(X_a,X_m) + X_m\ccancel{g(T, X_a)} - X_a\ccancel{g(T, X_m)})
                     = \frac{1}{2} \sum_a f' g^{ia} C_{am}\\
  \Gamma^{n+1}_{kl} & = \frac{1}{2} \sum_m g^{n+1,m}(X_kg(X_m, X_l) + X_lg(X_k,X_m) - X_mg(X_k,X_l))\\
                   & = \frac{1}{2} (X_k\ccancel{g(T, X_l)} + X_l\ccancel{g(X_k,T)} - Tg(X_k,X_l))\\
                   & = -\frac{1}{2} f' C_{kl}.
\end{align*}
This is true for all $1 \leq i, m, k, l \leq n+1$. 
From
\begin{align*}
  g(\nabla_{X_i}\nabla_TX_k, T) & = \frac{1}{4} (f')^2 \sum_{m,l,a} C_{kl} C_{am} g^{ia} g_{lm},
\end{align*}
we obtain
\begin{align}\label{eq:RiTkT}
  R^{(M,g)}(X_i, T, X_k, T)
  \nonumber & = g(\nabla_{X_i}\nabla_TX_k - \nabla_T\nabla_{X_i}X_k - \nabla_{\ccancel{[X_i, T]}}X_k, T)\\
            & = O(|f'|^2)\hat E_1 + O(|f''|)\hat E_2.
\end{align}

\noindent It remains to consider following term.
\begin{align*}
  R^{(M,g)}(X_i, X_j, X_k, T)
  & = g(\nabla_{X_i}\nabla_{X_j}X_k - \nabla_{X_j}\nabla_{X_i}X_k - \nabla_{\ccancel{[X_i, X_j]}}X_k, T).
\end{align*}
\noindent 
By \eqref{eq:ijkT} we only need to determine the following terms:
\begin{align*}
  g(\nabla_{X_j} X_k, \nabla_{X_i}T) & = g(\sum_l \Gamma^j_{kl} X_l, \sum_m \Gamma^i_{n+1,m}X_m)
                                       = \frac{1}{2} f' \sum_{l,a} \Gamma^j_{kl} g^{ia} C_{am} g_{lm}\\
  X_ig(X_k, \nabla_{X_j}T) & = X_i(\frac{1}{2}f'C_{jk})
                             = \frac{1}{2}f'X_i(C_{jk}).
\end{align*}
Thus we conclude that
\begin{align}\label{eq:RijkT}
  R^{(M,g)}(X_i, X_j, X_k, T) = O(|f'|) \tilde E.
\end{align}
Using the curvature tensor's symmetries, we have fully determined $R^{(M,g)}$ via \eqref{eq:Rijkl}, \eqref{eq:RiTkT} and \eqref{eq:RijkT} and find that
\begin{align*}
  R^{(M,g)}|_{(x_0, t_0)} = R^{(N \times [0,1], g_{f(t_0) + \mathrm{d} t^2})}|_{(x_0, 0)} + O(|f'|) E_1 + O(|f'|^2) E_2 + O(|f''|) E_3
\end{align*}
for some $E_1, E_2, E_3$, which only depend continuously on the path of metrics and its derivatives in $r$-direction.
\end{proof}

\printbibliography

\end{document}